\documentclass[11pt,reqno]{amsart}
\usepackage{Preambulo}
\usepackage{pdfpages}
 \usepackage[foot]{amsaddr}

\title[Time scheme for a chemotaxis-fluids model]
{A convergent time scheme for a chemotaxis-fluids model with potential consumption}
\author[D. Barbosa]{Daniel Barbosa \textsuperscript{1}}
\author[F. Guillén-González]{Francisco Guillén-González \textsuperscript{2}}
\author[G. Planas]{Gabriela Planas \textsuperscript{1}}

\address{\textsuperscript{1} Departamento de Matemática, Instituto de Matemática, Estatística e Computação Científica,
 Universidade Estadual de Campinas
}
\address{\textsuperscript{2} EDAN  and IMUS, Universidad de Sevilla}

\mathtoolsset{showonlyrefs=true}

\begin{document}

\begin{abstract}
    The present work deals with a Keller-Segel-Navier-Stokes system with potential consumption, under homogeneous Neumann boundary conditions for cell density and chemical signal, and of Dirichlet type for the velocity field, over a bounded three-dimensional domain. 
    The paper aims to develop a time discretization scheme converging to weak solutions of the system,
    which are uniformly bounded at infinite time.
    While global existence results are already known for simplified cases, either in absence of fluid flow or for linear consumption, the existence of global weak solutions for the fully coupled system with potential consumption has remained as an open problem.
\end{abstract}



\maketitle

\section{Introduction}\label{secao 1}

In this work, we investigate the dynamics of a chemotaxis-fluid system, a mathematical model with an intricate interplay between cellular behavior and fluid dynamics. The system, described within a bounded three-dimensional domain, encompasses three coupled nonlinear partial differential equations governing the evolution of cell density, chemical concentration,  
and fluid velocity field. 

Originating from biological scenarios such as bacterial motility in response to chemical gradients, chemotaxis modelling was first introduced by Keller and Segel \cite{Keller1970}. 
Chemotaxis is seen in many organic functions, even playing an important role in inflammatory diseases, wound healing, cancer metastasis or disease progressions, \cite{Murphy2001,Parales2000, Baker}.
In recent years, numerous studies have focused on chemotaxis-related models. For comprehensive reviews of this literature, see, for instance, \cite{Bellomo2015, Arumugam2020,Bellomo2022,MR4673380}. 

In cases where a chemical signal attracts cells that they also consume, a useful energy law emerges due to the interplay between the chemotactic and consumption terms. This effect has been explored in previous studies, allowing a priori estimates of possible solutions.  Global solutions have been found for three-dimensional domains in \cite{MR2860628}, and, with the addition of a logistic term on the cell density, the results were extended to higher-dimensional cases in \cite{MR3690294}. Some closely related considerations also lead to several other works such as \cite{MR3661546,MR3411100,MR4513095,MR4545829}. 

Recently,  an chemoattractive model with a potential consumption rate, 
was introduced in \cite{Correa_Vianna_Filho2023-gz} as
\begin{equation}
 \left\{ \ \begin{aligned}
 &n_t  =\Delta n-\nabla \cdot(n  \nabla c), \\
  &c_t =\Delta c-n^s c,
\end{aligned}\right.  
\end{equation}
where $n=n(t,x)\ge 0$ is the cell density  and $c=c(t,x)\ge 0$ the chemical concentration, for any $ x \in \Omega$ a bounded domain, and $ t >0$. The term $\nabla \cdot(n  \nabla c)$ corresponds to the chemotactic attraction and $n^s$ is the consumption rate of signal $c$, with $s \geq 1$. In \cite{Correa_Vianna_Filho2023-gz}, global weak solutions, uniformly bounded in time, are found in three-dimensional domains with minimal regularity, via a  continuous regularization of the problem. By the contrary,  a convergent time discretization is studied in \cite{Guillen-Gonzalez2023-gy} for the same problem.

In the presence of fluid flow, a buoyancy effect known as bioconvection occurs when bacterial populations exhibit collective behavior in an incompressible fluid, as discussed in \cite{MR2149639}. This effect, combined with chemotactic behavior may occur; this is the case, for example, where aerobic bacteria in water droplets search for oxygen \cite{PhysRevLett.93.098103}. This buoyancy effect couples cell density with a fluid velocity field, where additionally convective terms appear. This macro-scale chemotaxis fluid model has been extensively studied in various works such as \cite{MR3208807,MR2679718}.

In the present work, we focus on this kind of attraction-consumption chemotaxis models, in the presence of fluid flow by integrating the Navier-Stokes (NS) equations for an incompressible flow, for the velocity field $u=u(t,x)\in \mathbb{R}^3$ and pressure $P=P(t,x)\in \mathbb{R}$. This is achieved by introducing convective terms for both $n$ and $c$ equations and incorporating a source term into the NS equations, accounting for the gravitational effect of the heavier bacteria on the flow.  Moreover, homogeneous Neumann boundary conditions for both the cell
density $n$ and the chemical concentration  $c$ are considered, jointly with Dirichlet boundary conditions for the velocity field $u$. Then the system remains as 
\begin{equation}  \label{Sist cont}
 \left\{ \ \begin{aligned}
 &n_t+ u \cdot \nabla n  =\Delta n-\nabla \cdot(n  \nabla c), \\
 &c_t  + u \cdot \nabla c= \Delta c- n^s c, \\
 &u_t +u \cdot \nabla u  +\nabla P  =\Delta u + n\nabla \Phi,  \\
 &\nabla \cdot u  =0, 
 \\ 
 &\left.\partial_\eta n\right|_{\Gamma}=\left.\partial_\eta c \right|_{\Gamma}= \left. u\right|_{\Gamma}= 0 , \quad n(0) = n^0, \quad c(0) = c^0, \quad u(0) = u^0.
\end{aligned}\right.  
\end{equation}
Here, $ \eta$ represents the outward normal vector to the boundary $ \Gamma$ of the domain $ \Omega$.
The gravitational potential, $\Phi=\Phi(x)$, is assumed to be a  given $W^{1,\infty}(\Omega)$ function. Finally, $n^0, c^0, u^0$ are the initial data.

The case  $s=1$ has been treated first in convex domains, obtaining global classical solutions in two-dimensional domains \cite{Winkler2012}; and after were extended to non-convex domains in \cite{Jiang2015}. 
Global weak solutions have also been obtained in three-dimensional domains in \cite{Winkler2016}.

The objective of this work is twofold: on the one hand to extend the existing results of \cite{Winkler2016} to the case of potential consumption rate $ n^s $ with $s > 1$ and through a convergent time discrete scheme, and on the other hand extending the results of the system without fluid given in  \cite{Guillen-Gonzalez2023-gy}  to the coupled chemotaxis-fluid system.
More specifically, the main contributions of this paper are the following:
\begin{enumerate}
    \item[(i)] 
    Getting the convergence of a time discrete scheme towards weak solutions of system \eqref{Sist cont}, using a new truncation operator for the chemotaxis and consumption terms and also truncating the source term from the Navier-Stokes system. 
    
    \item[(ii)]  Proving the existence of global weak solutions for the chemotaxis-Navier-Stokes model with potential consumption $n^sc$ for $s > 1$ which are uniformly
bounded in time. 
\end{enumerate}

It is worth mentioning that although the numerical simulation of chemotaxis models is a significant and growing area of research, there are relatively few studies focused on the numerical approximation for chemoattraction-consumption
models. Beyond \cite{Correa_Vianna_Filho2023-gz}, to the best of our knowledge, the case $s=1$ has been addressed in only two studies: \cite{Guillen-Gonzalez2024-nb} and \cite{Duarte-Rodriguez2021-wx}. The former investigates several finite element schemes, while the latter develops numerical approximations for the system coupled with the Navier–Stokes equations.
This shortage may also be linked to the complex interaction between chemotaxis and consumption effects, as well as the difficulty for adapting theoretical frameworks into numerical approaches. 
Moreover, constructing convergent schemes is particularly demanding, as it requires carefully balancing the competition effects between  chemoattraction and consumption in the derivation of energy estimates.

This paper is organized as follows. Section \ref{sec main} presents the main result and the time discrete scheme. For that, we introduce a new truncation operator and the corresponding time discrete regularized system of \eqref{Sist cont}. Then 
 the existence of global weak solutions of \eqref{Sist cont} is stated. 
Section \ref{secao 2} focuses on preliminary results that are either referenced or proved. Keys are 
Lemma~\ref{lema GNS produto}, which deals with the source term $n \nabla \Phi$ of the fluid system and 
Lemma \ref{Lema dif T G} clarifies the relationship between the direct upper truncation $T_0^m$ and the new one $G_0^m$, given in \eqref{T0m}-\eqref{G0m} below.
Section \ref{secao 3} treats the existence of the time discrete regularized problem which is obtained by a fixed point argument.
Section \ref{secao 4} presents some a priori estimates based on an adequate energy inequality. This is done separately for $s \in (1,2)$ and $s \geq 2$ because the energy becomes singular with respect to cell variable $n$ for the case $s<2$.
 Section \ref{secao 5} shows the passage to the limit of the time discrete regularized system towards weak solutions of \eqref{Sist cont}. 

\section{Main Theorem and time discrete scheme} \label{sec main}
Here and henceforth $\Omega$ is assumed to be a bounded domain in $\mathbb{R}^3$ with boundary  $\Gamma$ sufficiently regular as specified below.
For simplicity, we will omit the domain in the notation of functional spaces. That is, we will denote, for instance, $L^2(\Omega)$ simply as $L^2$ from now on.

We start by introducing some usual spaces in fluid problems. Let 
$$
\mathcal{V} \coloneqq \{ v \in C^\infty_0(\Omega)^3, \nabla \cdot v = 0\}.
$$
Denoting by  $ \overline{\mathcal{V}}^{X}$ the closure of $\mathcal{V}$ in the $X$-norm, we consider
\begin{align}
 &H \coloneqq \overline{\mathcal{V}}^{L^2} = \{ v \in L^2, \nabla \cdot v = 0, \left.v\cdot \eta\right|_{\Gamma} = 0 \}, \\
 & V \coloneqq \overline{\mathcal{V}}^{H^1_0} = \{ v \in H^1_0, \nabla \cdot v = 0\}.
\end{align}
 Both identifications hold true for bounded Lipschitz domains \cite{Boyer2012-pq}. Moreover,  one has the Gelfand Triple
$
V \subset H \subset V',
$
with compact and dense embeddings.
We also define $W^{1,p}_{0,\sigma} \coloneqq \overline{\mathcal{V}}^{W_0^{1,p}}$ for $1<p<\infty$.

The main properties of problem \eqref{Sist cont} are nonnegative variables $n,c\ge 0$, $n$-conservation, $c$-pointwise estimates and an energy inequality letting appropriate estimates of all variables $n,c$ and $u$.  

We are going to impose conditions on the regularity of the domain $\Omega$ such that the boundary terms on the energy inequality associated to problem \eqref{Sist cont}  are dealt with.  Specifically, we impose the following hypotheses:

\

\noindent
{\bf (H2)-regularity:}  
{\sl For any $f \in L^2$, there exists a unique $z \in H^{2}$ solution to the Poisson-Neumann problem
\begin{equation} \label{poisson neuman}
-\Delta z+z  =f \quad \text { in } \Omega, \quad
\left.\partial_\eta z \right|_{\Gamma} =0.
\end{equation}
Moreover, there exists $C>0$ such that 
$$
\|z \|_{H^{2}} \leq C\|-\Delta z+z\|_{L^2}.
$$
}

\noindent
{\bf (H2)-approx.:}  {\sl
    For any $z \in H^2$ with $\left.\partial_\eta z \right|_{\Gamma} =0$, there is a sequence $ \left\{\rho_n\right\} \subset C^2(\overline{\Omega}) $ such that $\left.\partial_\eta \rho_n \right|_{\Gamma} =0$ and  $\rho_n \rightarrow z $ in $ H^2$.
}

\

Note that, as shown in \cite{Correa_Vianna_Filho2023-gz}, it 
suffices require $\Gamma$ to be at least $C^{1,1}$ to both hypotheses hold.

\

We now introduce what is meant by a weak solution to problem \eqref{Sist cont}.

\begin{mydef}[Weak solution] \label{def sol}
    A triplet $(n,c,u)$ is a weak solution in $(0,\infty)\times \Omega$ of system \eqref{Sist cont} if the following features hold:
    \begin{itemize}
\item  (nonnegative) $n,c\ge 0$ \ a.e. in $(0,\infty)\times \Omega$, 
\item  ($n$-conservation)
$$\int_\Omega n(t,\cdot) = \int_\Omega n^0 \quad \text {a.e. } t \in(0, \infty),  $$
\item ($c$-pointwise estimates)
$$0 \leq c(t,x) \leq \norm{c^0}_{L^\infty } \quad\text {a.e. } (t,x) \in(0, \infty)\times \Omega, $$
\item ($c$-weak estimates)
$$
\nabla c\in L^2(0,\infty;L^2),
$$
\item (energy regularity) 
\begin{align*}
   n & \in L^\infty(0,\infty;L^s) \cap L^{5s/3}_{loc}(0,\infty;L^{5s/3}),\\
  c &\in L^\infty(0, \infty ; H^1) \cap L^2_{loc}(0, \infty ; H^2), \quad \nabla c \in L^4_{loc}(0, \infty ; L^4),
    \\
 u &\in L^\infty(0,\infty;H) \cap L^2_{loc}(0,\infty;V),
\end{align*} 
\item (flux regularity)    
$$ n \nabla c,\    \nabla n \in  
 L^{5s/(3+s)}_{loc}(0, \infty; L^{5s/(3+s)}),
 \quad \hbox{for $s \in (1,2)$,}
$$
$$
 n\nabla c,\   \nabla n \in L^2_{loc}(0,\infty;L^2),
  \quad \hbox{for $s \in [2,\infty)$,}
$$
 
\item
and $(n,c,u)$ satisfies
    \begin{align*}
        \int_\Omega n_t  \phi  - \int_\Omega n\,  u \cdot\nabla \phi   + \int_\Omega \nabla n \cdot\nabla \phi  &=\int_\Omega n   \nabla c  \cdot \nabla \phi  ,\quad \forall\, \phi \in W^{1,10}, \text { a.e. } t \in(0, \infty),  \\
 c_t +u \cdot \nabla c  -\Delta c &= -n^s c, \quad \text { a.e. }(0, \infty) \times \Omega, \\
 \int_\Omega u_t  \varphi  +\int_\Omega (u \otimes u) \nabla \varphi  
 + \int_\Omega \nabla u\cdot \nabla \varphi  &= \int_\Omega n \nabla \Phi\cdot \varphi , \quad \forall \,\varphi \in 
 W^{1,5/2}_{0,\sigma},  \text { a.e. } t \in(0, \infty),
    \end{align*}
and the initial conditions $(n,c,u)|_{t=0}=(n^0,c^0,u^0)$. Hereafter, $(u \otimes u)_{ij} = u_iu_j$.
\end{itemize}
\end{mydef}


\begin{remark}
    With the  regularity given in previous definition and looking at the system \eqref{Sist cont}, the following time derivative regularity hold:
    $$n_t \in L^{10/9}_{loc}(0,\infty;(W^{1,10})')\quad \hbox{for  $s<2$,} 
    \quad n_t \in L^{5/3}_{loc}(0,\infty;(W^{1,5/2})') \quad \hbox{for $s\geq 2$,}
    $$
    $$
    c_t \in L^{5/3}_{loc}(0,\infty;L^{5/3}),
    \quad \hbox{and} \quad 
    u_t \in 
    L^{5/3}_{loc}(0,\infty;(W^{1,5/2}_{0,\sigma})^\prime).
    $$
     In particular, the initial conditions $(n,c,u)|_{t=0}=(n^0,c^0,u^0)$ have  sense.
\end{remark}

We may now state our main result.
\begin{theorem} \label{resultado principal}

    Let $\Omega \subset \mathbb{R}^3$ be a bounded domain satisfying 
    {\bf (H2)-regularity} and  {\bf (H2)-approx.}. 
    Let $n^0\in L^s$,  $c^0\in H^1 \cap L^\infty$ with $n^0\ge0, c^0\ge 0$ a.e.\ in $\Omega$, and $u^0 \in H$. 
Then, one has a time discrete problem (see \eqref{Sist Discretizado} below) which is convergent towards   a weak solution $(n,c,u)$ in $(0,\infty)\times \Omega$ of  problem \eqref{Sist cont}.
\end{theorem}

\subsection{Definition of the time discrete scheme}

A convenient way to rewrite the equations is by setting an adequate change of variables, changing the chemical variable $c$ by the auxiliary variable 
$z = \sqrt{c + \alpha^2}$ for some $\alpha >0$, which simplifies the test functions involved for obtaining an energy law of system \eqref{Sist cont}. In fact,  $c$-equation $\eqref{Sist cont}_2$ is rewritten as  
$$
z_t+u\cdot\nabla c -\Delta z -\frac{|\nabla z|^2}{z}
= -\frac12 n^s \left( z-\frac{\alpha^2}z \right).
$$

A time discretization approach is considered in \cite{Guillen-Gonzalez2023-gy} in the case of the corresponding chemotaxis model, but without the interaction with a fluid flow. A regularization to this system is made with upper and lower truncations in order to have bounded terms and prevent division by zero, respectively. 
For that reason, we define the lower truncation for $z$ as
\begin{equation} \label{Talpha}
    T_\alpha(z)=\left\{\begin{array}{cl}
\alpha & \text { if } z \leq \alpha, \\
z & \text { if } z \geq \alpha,
\end{array}\right.
\end{equation}
and the lower-upper truncation for $n$,

\begin{equation} \label{T0m}
   T_0^m(n)=\left\{ \ \begin{array}{cl}
0 & \text { if } n \leq 0, \\
n & \text { if } n \in [0, m], \\
m & \text { if } n \geq m .
\end{array}\right. 
\end{equation}

In this work, a new regularization is proposed.  Compared with 
\cite{Guillen-Gonzalez2023-gy}, this regularization fits better concerning the test functions employed in the $n$ and $z$ equations. To do that, we define the new lower-upper truncation operator but for $n^s/s$,
\begin{equation} \label{G0m}
G^m_0 (n) = \left\{ \ \begin{array}{ll}
 0 &\text{ if } n\leq 0, \\
 \frac{1}{s}n^{s} & \text{ if } n \in [0, m], \\
 \frac{m}{s-1} n^{s-1} - \frac{m^s}{s(s-1)} & \text{ if } n\geq  m .
\end{array}\right. 
\end{equation}
In fact,  $G^m_0 (n)$ is defined such that $G^m_0 \in C^1(\mathbb{R})$ and satisfying the important equality
\begin{equation} \label{eq Gm Tm}
    (G^m_0)' (n) = T^m_0 (n) n^{s-2} 1_{\{n \geq0\}}.  
\end{equation}

At this point, 
we consider the  truncated  
$(n,z,u)$-system as follows
\begin{equation}   \label{Sist cont reg}
 \left\{ \ \begin{array}{l}
 n_t+ u \cdot \nabla n  -\Delta n + \nabla \cdot(T_0^m(n)  2T_\alpha(z)\nabla z)=0, \\
 \displaystyle z_t  + u \cdot \nabla z- \Delta z - \frac{\abs{\nabla z}^2}{T_\alpha(z)} 
 =- \frac{s }{2}G^m_0(n)\left(z - \frac{\alpha^2}{T_\alpha(z)} \right), \\
 u_t +u \cdot \nabla u  -\Delta u +\nabla P  = T_0^m(n)\nabla \Phi,  \\
 \nabla \cdot u  =0, \\ 
  \left.\partial_\eta n\right|_{\Gamma}=\left.\partial_\eta z \right|_{\Gamma}= \left. u\right|_{\Gamma}= 0 , \quad n(0) = n^0, \quad z(0) = z^0, \quad u(0) = u^0,
\end{array}\right.  
\end{equation}
where $z^0 = \sqrt{c^0 + \alpha^2}$.

As for the time discretization, we will divide the interval $[0, \infty)$ into subintervals denoted by $I_i=\left[t_{i-1}, t_i\right)$, with $t_0=0$ and $t_i=t_{i-1}+k$, where $k>0$ is the time step. We use the notation
$$
\delta_t n^i=\frac{n^i-n^{i-1}}{k}, \ \forall i \geq 1,
$$
for the discrete time derivative of a sequence $n^i$.
Then, we define the following Backward Euler time scheme of truncated system \eqref{Sist cont reg} with semi-implicit approximation of the convection of the fluid: given $(n^{i-1},z^{i-1},u^{i-1})$, to find $(n^{i},z^{i},u^{i},P^i)$ solving
\begin{equation}  
  \label{Sist Discretizado} 
 \left\{ \ \begin{aligned}  
&\delta_t n^i+u^i \cdot \nabla n^i  
-\Delta n^i
+\nabla \cdot(T_0^m(n^i)  2T_\alpha(z^i)\nabla z^i)=0, \\
&\delta_t z^i+u^{i} \cdot \nabla z^i 
-\Delta z^i
- \frac{\abs{\nabla z^i}^2}{T_\alpha(z^i)}
=-\dfrac{s}{2} G^m_0(n^i)\left( z^i - \dfrac{\alpha^2} {T_\alpha(z^i)}  \right) 
,  \\
&\delta_t  u^i
+u^{i-1} \cdot \nabla u^i 
-\Delta u^i
+\nabla P^i
 =
T_0^m(n^{i}) \nabla \Phi,  \\
 &\nabla \cdot u^i =0, 
\\ 
  &\left.\partial_\eta n^i \right|_{\Gamma}=\left.\partial_\eta z^i \right|_{\Gamma}
  = \left. u^i \right|_{\Gamma}= 0.
\end{aligned}\right.
\end{equation}
In fact, \eqref{Sist Discretizado} is a fully coupled and nonlinear time scheme.

Since $z^0 = \sqrt{c^0 + \alpha^2}\ge \alpha$ and $n^0\ge 0$ in $\Omega$, we will determine that  $z^i \geq \alpha$ and $n^i \geq 0$ in $(0,\infty)\times \Omega$, hence we will drop the lower truncations of $z^i$ by $\alpha$ and $n^i$ by $0$ 
in \eqref{Sist Discretizado}. Thus, to return to the original system \eqref{Sist cont}, we will check what happens to the solution of \eqref{Sist Discretizado} as $(m,k) \to (\infty,0)$.

In order to get energy estimates (see section \ref{est lin por partes}) we will need to impose more regular initial data for the time discrete problem  \eqref{Sist Discretizado} than $(n^{0}, z^{0}, u^{0})\in L^s\times (H^1\cap L^\infty) \times H$ imposed in Theorem \ref{resultado principal}. In fact, we consider a sequence 
$(n^{0}_j, z^{0}_j, u^{0}_j)$ satisfying the following features: 
\begin{itemize}
\item   (regularity)  
\begin{equation}  \label{reg n0}  
 n^0_j \in   H^1\cap L^s,
 \quad  z^0_j  \in  H^2 ,
 \quad u^0_j \in V ,
\end{equation}
\item (nonnegative and conservation) 
\begin{equation}  \label{compat-0}  
n^{0}_j \ge 0, \ z^{0}_j \ge \alpha \ \text{ in $\Omega$},
\quad \int_\Omega n^{0}_j = \int_\Omega n^{0},
\end{equation}
\item (approximation) 
\begin{equation}  \label{approx-0}  
(n^{0}_j, z^{0}_j, u^{0}_j) \to (n^{0}, z^{0}, u^{0})\quad   \text{in $L^s\times (H^1\cap L^\infty) \times H$, as $j\to \infty$}.
\end{equation}
\end{itemize}
Finally, to get that all estimates of this paper be $j$-independent, it suffices to  consider for each $j$ a small  enough time step $ k(j)$ such that the following estimates hold:
  \begin{equation}  \label{reg u0}
  \begin{aligned}
  k (j)\, \| n^0_j \|_{W^{1,5s/(s+3)} }^{5s/(s+3)}\le 1 \ \text{ if } s\le 2, 
& \qquad 
 k(j) \, \|n^0_j \|_{H^1}^2 \le 1 \ \text{ if } s\ge 2, \\
  k(j)\, \| z^0_j\|_{H^2}^2 \le 1, &
   \qquad k(j)\, \| u^0_j\|_{V}^2 \le 1.  \end{aligned}
\end{equation}
 By simplicity in the notation, from now on we omit the index $j$.

\section{Preliminary Results}\label{secao 2}

This section presents several technical results that will be essential tools throughout the paper.

Lemmas \ref{lema d2 z} through \ref{lema conv raiz}, previously established and proven in \cite{Correa_Vianna_Filho2023-gz,Guillen-Gonzalez2023-gy}, are restated here for completeness.
%
%
The following Lemma is used to obtain an energy inequality without a boundary term otherwise present.
\begin{lema} \label{lema d2 z}
   Suppose  {\bf (H2)-regularity} and  {\bf (H2)-approx.}. Let $z \in H^2$ be such that $\left.\partial_\eta z\right|_{\Gamma}=0 $ and $z \geq \alpha$ for some $\alpha>0$. Then there exist positive constants $C_1, C_2>0$, independent of $\alpha$, such that
$$
\int_{\Omega}|\Delta z|^2 d x+\int_{\Omega} \frac{|\nabla z|^2}{z} \Delta z d x \geq C_1\left(\int_{\Omega}\left|D^2 z\right|^2 d x+\int_{\Omega} \frac{|\nabla z|^4}{z^2} d x\right)-C_2 \int_{\Omega}|\nabla z|^2 d x .
$$
\end{lema}

As for the time discretization, when dealing with time discrete derivatives the following holds.
\begin{lema} \label{lema derivada discreta}
     Let $z^n, z^{n-1} \in L^{\infty}$, and let $f: \mathbb{R} \rightarrow \mathbb{R}$ be a $C^2$ function.
Then
$$
\int_{\Omega} \delta_t z^n(x) f'(z^n(x)) dx =\delta_t \int_{\Omega} f(z^n(x)) d x+\frac{1}{2 k} \int_{\Omega} f''(b^n(x))\left(z^n(x)-z^{n-1}(x)\right)^2 d x,
$$
where $b^n(x)$ is an intermediate point between $z^n(x)$ and $z^{n-1}(x)$.
\end{lema}

To establish the convergence of the powers of a function, we will utilize the following result.

\begin{lema}\label{lema conv raiz}
    Let $p \in(1, \infty)$, and let $\left\{w_n\right\}$ be a sequence of nonnegative functions such that $w_n \rightarrow w$ in $L^p(0, T ; L^p)$ as $n \rightarrow \infty$. Then, for every $r \in(1, p), w_n^r \rightarrow w^r$ in $L^{p / r}(0, T ; L^{p / r})$ as $n \rightarrow \infty$.
\end{lema}

We also make use of 
the classical version of Aubin-Lions theorem for compactness \cite{Boyer2012-pq}.
\begin{lema} \label{Lema Aubin Lions}
    Let $X$, $B$ and $Y$ be three Banach spaces with $X \subset B \subset Y$. Suppose that $X$ is compactly embedded in $B$ and  $B$ is continuously embedded in $Y$.
For $1 \leq p, q \leq \infty$, let
$$
W=\left\{u \in L^p(0, T ; X) \mid \partial_t{u} \in L^q(0, T ; Y)\right\} .
$$
\begin{itemize}
    \item[i)] If $p<\infty$ then the embedding of $W$ into $L^p(0, T ; B)$ is compact.
    \item[ii)] If $p=\infty$ and $q>1$ then the embedding of $W$ into $C([0, T] ; B)$ is compact.
\end{itemize}
\end{lema}




We present a discrete version of uniform in time Gronwall's estimates.
\begin{lema} \label{Lema Gronwall Discreto}
Consider a sequence of inequalities
$$
\delta_t a^n + \lambda \, a^n \leq C , \quad \forall\, n\ge 1,
$$
where $\lambda >0$ and $\{a^n\}_{n \in \mathbb{N}} \subset \mathbb{R}_+$. 
Then, 
 the following estimate holds
$$
a^n \leq (1+\lambda k)^{-n} a^0 + \frac{C}{\lambda} \left(1 - (1+\lambda k)^{-n} \right),
\quad \forall\, n\ge 1.
$$
In particular, 
$$
a^n \leq a^0 + \frac{C}{\lambda}, \quad \forall\, n\ge 1.
$$
\end{lema}
\begin{proof}
We adapt  a proof given in \cite{Emmrich2000DiscreteVO}. 
    Let $\tilde{a}^n \coloneqq (1+\lambda k)^{n} a^n$. Then, from the recursive inequality 
    $$
\delta_t \tilde a^n \leq (1+\lambda k)^{n-1} C , \quad \forall\, n\ge 1,
$$
Now, summing over $n$ leads to
\begin{equation*}
    \tilde{a}^n \leq  \tilde{a}^0 + C \, k \sum_{j=0}^{n-1} (1+\lambda k)^{j} 
    = a^0 + \frac{C}{\lambda } (  (1+\lambda k)^{n} - 1),
\end{equation*}
and finally
$$
a^n \leq(1+\lambda k)^{-n} a^0
+ \frac{C}{\lambda}\left(1-(1+\lambda k)^{-n}\right) . 
$$
\end{proof}
\begin{lema} \label{lema GNS produto}
    Let $\Omega \subset \mathbb{R}^d$ be a bounded Lipschitz domain, with $d \leq 3$. Let also $f,g$ be functions defined on $\Omega$ such that 
   $$\int_\Omega  |f| = K, \quad   |f|^{s/2} \in H^1 \ (s>1)\quad \text{and}\quad  g \in H^1_0.
   $$
    Then for any $\epsilon_1>0$ and $\epsilon_2>0$ there exists $C>0$  such that
    $$
    \int_\Omega \abs{f}\abs{g} \leq  \epsilon_1 \norm{\nabla g}_{L^2}^2 
    + \epsilon_2 \norm{\nabla |f|^{s/2}}_{L^2}^2 + C,
    $$
    where $C$ depends on $\epsilon_1$, $\epsilon_2$ and $\int_\Omega  |f|$.
\end{lema}
This lemma will be used to deal with the source term $T_0^m(n^{i}) \nabla \Phi$ on the fluid velocity equation, when it is multiplied by $u^i$, by taking $f=T_0^m(n^{i})$ and $g=u^i$.

\begin{proof}
  By using the embedding $H^1 \hookrightarrow L^6$, the   equivalence between $\|g\|_{H^1}$ and  $\|\nabla g\|_{L^2}$ for any $g\in H^1_0$ and Young's inequality, 
       $$
       \int_\Omega \abs{f}\abs{g} \leq  \norm{f}_{L^{6/5}} \norm{g}_{L^6}
       \le C \norm{|f|^{s/2}}^{2/s}_{L^{12/(5s)}}  \norm{ \nabla g}_{L^2}
    \le
    C \norm{|f|^{s/2}}^{4/s}_{L^{12/(5s)}} 
       +  \epsilon_1 \norm{ \nabla g}_{L^2}^2.
       $$
       Since $2/s<12/5s<6  $, by interpolating
  $$ 
        \norm{|f|^{s/2}}^{4/s}_{L^{12/(5s)}} \leq \left( \norm{|f|^{s/2}}^a_{L^6} \right)^{4/s} \left( \norm{|f|^{s/2}}^{(1-a)}_{L^{2/s}} \right)^{4/s}  
      =\norm{|f|^{s/2}}^{4a/s}_{L^6} \left(\int_\Omega \abs{f} \right)^{2(1-a)}
    $$
    where  $a = s/2(3s-1)$. Then, $4a/s=2/(3s-1)<1$ for any $s>1$, so applying Young's inequality and again the embedding $H^1 \hookrightarrow L^6$,
      \begin{align}
       \int_\Omega \abs{f}\abs{g} 
 &\leq \epsilon \norm{|f|^{s/2}}_{L^6}^2 + C(K) +  \epsilon_1 \norm{ \nabla g}_{L^2}^2
\nonumber  \\ \label{interm}
  &\le \epsilon \norm{\nabla |f|^{s/2}}_{L^2}^2 
      + \epsilon \norm{ |f|^{s/2}}_{L^2}^2
      + C(K) +  \epsilon_1 \norm{ \nabla g}_{L^2}^2.
\end{align}
 By using  Poincar\'e's inequality we have
      \begin{equation} \label{Poinc}
       \norm{|f|^{s/2}}_{L^2}^2 \leq C \norm{\nabla |f|^{s/2}}^2_{L^2} 
       + C \norm{|f|^{s/2}}^2_{L^1} .
       \end{equation}
      The remaining term $\norm{|f|^{s/2}}^2_{L^1}$ is dealt separately in the cases $s\leq 2$ and $s>2$. Indeed, if $s \leq 2$, then 
      $$
      \norm{|f|^{s/2}}_{L^1} ^2 =\left(\int_\Omega |f|^{s/2}\right)^2 \leq \norm{f}_{L^1}^{s}\abs{\Omega}^{(2-s)} \leq C(K).
      $$
     If $s>2$, by interpolation and Young inequalities,   
       \begin{align*}
        \norm{|f|^{s/2}}_{L^1}  ^2
        = \norm{f}^{s}_{L^{s/2}} \leq \norm{f}_{L^s}^{sa} \norm{f}_{L^1}^{s(1-a)} 
        \leq \epsilon \norm{f}^{s}_{L^s} + C(K)
        = \epsilon \norm{|f|^{s/2}}_{L^2}^2 + C(K).
       \end{align*}
      Therefore, in both cases, from \eqref{Poinc}, we arrive at
      \begin{equation*} 
       \norm{|f|^{s/2}}_{L^2}^2 \leq C \norm{\nabla |f|^{s/2}}^2_{L^2} 
       + C (K),
       \end{equation*}
      hence using \eqref{interm} the proof is finished.
 \end{proof}     
%
%
%

    We also make use of a comparison inequality between the truncation operators.
    \begin{lema} \label{Lema dif T G}
    Let $T_0^m$ and $G_0^m$ be the truncations defined in \eqref{T0m} and \eqref{G0m}. Then for every $s>1$ and $q \geq 1$, 
     it holds that
    $$
    n^{s-q} \, T_0^m(n)^q \leq s \, G_0^m(n),\quad \forall\, n\in \mathbb{R}. 
    $$

    \end{lema}
    \begin{proof}
        If $0\le n \leq m$, the equality becomes  $n^s= s\, G^m_0(n)$ which is true looking at  \eqref{T0m} and \eqref{G0m}, and if $n<0$, the equality reduces to  $0=0$. Let $n > m$ and define the auxiliary function,
     $$
     f(n) \coloneqq  \frac1s n^{s-q} \, T^m_0(n)^q  - G^m_0(n)
     = \frac{1}{s} n^{s-q} m^q - \frac{m}{s-1}n^{s-1} + \frac{m^s}{s(s-1)}.
     $$ 
     It suffices to prove that $f(n)\le 0 $ for all $n>m$. 
    Note that 
    \begin{align*}
        \lim_{n \to m^+} f(n) &= \frac{1}{s} m^s - \frac{m^s}{s-1} + \frac{m^s}{s(s-1)}  = 0.
    \end{align*}
 On the other hand, for any $n>m$ 
    \begin{align*}
        f'(n) &= \frac{s-q}{s} n^{s-q-1}m^q - mn^{s-2}
              = \frac{s-q}{s} n^{s-q-1}m^q - n^{s-q-1}  n^{q-1} m \\
              &\leq \frac{s-q}{s} n^{s-q-1}m^q - n^{s-q-1}m^q 
              = n^{s-q-1}m^q \left( \frac{s-q}{s} - 1 \right) \leq 0,
    \end{align*}
 because as $q \geq 1$ it holds $(s-q)/{s}  \leq 1$. In summary, $f(n)$ is a non-increasing negative function for all $n>m$.
  \end{proof}

\section{Existence of solutions for the time discrete scheme 
} \label{secao 3}


\begin{prop}
            Let $\Omega$ be as in Theorem \ref{resultado principal}. Let $(n^{i-1},  z^{i-1},u^{i-1}) \in  (L^2\cap L^s) \times L^\infty \times (H \cap L^5)$, with $n^{i-1} \geq 0, \ z^{i-1} \geq \alpha$ where $z^{i-1} = \sqrt{c^{i-1} + \alpha^2}$. Then for any $\alpha$ small enough, there exists $(n^{i},z^{i},u^{i},P^{i}) \in  H^2 \times H^2 \times (H^2\cap V) \times H^1$ solving the time discrete scheme \eqref{Sist Discretizado}.
\end{prop}

\begin{proof}
    The existence of solutions is proven by a fixed-point argument. 
    Given $(\overline{n},\overline{z},\overline{u}) \in (W^{1,4})^2 \times V$, we consider the (decoupled) auxiliary problem to define $(n,z,u,P)$ solving
    \begin{equation}
        \left\{\begin{aligned} 
& \dfrac{u}{k}
 +(u^{i-1}  \cdot \nabla) u 
 - \Delta u 
 + \nabla P
  =   T_0^m(\overline{n}) \nabla \Phi 
 + \frac{u^{i-1}}{k}, 
 \quad \nabla\cdot u=0,
 \\
&\dfrac{z}{k} 
- \Delta z 
+ \dfrac{s}{2}G^m_0(\overline{n}) z 
= \alpha^2 \dfrac{s}{2}  \dfrac{G^m_0(\overline{n})}{ T_\alpha(\overline{z})} 
+ \dfrac{\abs{\nabla \overline{z}}^2}{T_\alpha(\overline{z})} 
+ \dfrac{z^{i-1}}{k} 
-\overline{u} \cdot \nabla \overline{z}, \\
 &\dfrac{n}{k} +
\overline{u}   \cdot \nabla n 
- \Delta n = 
-2\, \nabla \cdot(T_0^m(\overline{n})  T_\alpha(\overline{z}) \nabla z) 
+\dfrac{n^{i-1}}{k}, \\
 &  \left.\partial_\eta n \right|_{\Gamma}=\left.\partial_\eta z \right|_{\Gamma}= \left. u \right|_{\Gamma}= 0.
\end{aligned}\right.
    \end{equation}
We first ensure, through the Lax-Milgram theorem, the existence of $u$ and $z$, and then the existence of  $n$ (depending on $z$), in such a way that a compact map 
$$S: (\overline{n},\overline{z},\overline{u}) \to (n,z,u)$$ is defined from $(W^{1,4})^2\times V$ to itself. Finally, by Leray-Schauder fixed-point theorem \cite{Gilbarg2001}, we will obtain the desired solution of the discrete scheme \eqref{Sist Discretizado}. 

\

{\sl Step 1. Operator $S$ is well-defined}
 
Given $\overline{n}\in W^{1,4}$, the existence and uniqueness of the pair $(u,P$) follow a standard argument, by first applying Lax-Milgram theorem to get only $u\in V$, and after via De Rham's Lemma and the regularity of the Stokes problem, in a bootstrapping argument, getting that there exists a unique $(u,P) \in H^2  \times (H^1 \cap L_0^2)$, where $ L_0^2 = \{ f \in L^2 : \int_\Omega f =0\}$.
 
The existence and uniqueness of $z\in H^1$ follow from the Lax-Milgram theorem,  obtaining that in fact $z \in H^2$ via regularity of the Poisson-Neumann problem \eqref{poisson neuman}. From that, looking at $n$ equation and applying Lax-Milgram theorem, we again obtain, through a bootstrapping argument, necessary due to the presence of the convective term $\bar{u}   \cdot \nabla n$, that there exists a unique $n \in H^2$. 

The bounds obtained  throughout the bootstrapping argument are sufficient to conclude that $S$ is a well-defined compact mapping, defined by the composition of a sequentially continuous map from $(W^{1,4})^2\times V$ into $(H^2)^2\times (H^2\cap V)$ and the  compact embedding from $(H^2)^2\times (H^2\cap V)$ into  $(W^{1,4})^2\times V$.


{\sl Step 2. A priori estimates of possible fixed-points}

To apply the fixed point theorem, we now need $\lambda$-independent estimates for all possible  triplets $(n,z,u)$ in $(W^{1,4})^2\times V$ satisfying $(n,z,u)= \lambda \, S(n,z,u)$ for some $\lambda \in [0,1]$. 

The case  $\lambda = 0$ is trivially satisfied because $(n,z,u) \equiv (0,0,0)$. 
Let $\lambda \in (0,1]$, therefore if $(n,z,u) = \lambda \, S(n,z,u)$, then $(n,z,u)$ satisfies the equations 

\begin{equation}  
\tag{$1_\lambda$} \label{1lamb} \ 
\dfrac{u}{k} 
+ u^{i-1} \cdot \nabla u 
- \Delta u 
+ \nabla P
 = 
 \lambda \,T_0^m(n) \nabla \Phi + \lambda \frac{u^{i-1}}{k},
\end{equation}
\begin{equation}  
\tag{$2_\lambda$} \label{2lamb}
   \dfrac{n}{k} 
    + u \cdot \nabla n - \Delta n = 
    -\nabla \cdot(T_0^m(n)2T_\alpha(z)   \nabla z) 
    +\lambda\frac{ n^{i-1}}{k},
\end{equation} 
 \begin{equation} 
\tag{$3_\lambda$} \label{3lamb}
     \dfrac{z}{k \lambda} 
+u \cdot \nabla z 
- \Delta \frac{z}{\lambda}
+ \frac{s}{2 \lambda}G_0^m(n) z 
= \alpha^2 \frac{s}2 \frac{G_0^m(n)}{T_\alpha(z)} 
+ \frac{\abs{\nabla z}^2}{T_\alpha(z)} 
+ \frac{z^{i-1}}{k}  .
 \end{equation}  
 
 Firstly, as $u \in V$, we can test \eqref{1lamb} by $u$, obtaining
\begin{align*}
    \dfrac{1}{k} \int \abs{u}^2  
    +\int \abs{\nabla u}^2   
    &= 
     \lambda \int T_0^m(n) \nabla \Phi \cdot u + \int \lambda \frac{u^{i-1}}{k} \cdot u  \\
    &\leq \delta \int \abs{u}^2 
    + C_\delta \left( \frac{1}{k^2} \norm{u^{i-1}}_{L^2}^2 
    + m^2 \norm{\nabla \Phi}_{L^\infty}^2 \right).
\end{align*}
By taking $\delta<1/k$ there follows that $u$ is bounded in $V$  (independently of $\lambda$).

We now deal separately with $n$ and $z$. The proof is analogous to that presented in  \cite[Theorem 3.1]{Guillen-Gonzalez2023-gy}, and here we only point out the main differences.

The first step is to show that 
$n$ is a nonnegative function. For that, $n$ equation \eqref{2lamb} is tested by  the negative part of $n$, defined as $n_- \coloneqq \min\{ 0,n\}$. Noting that for the convective term 
 $$
 \int (u \cdot \nabla n) n_- =   \frac{1}{2}
\int u \cdot \nabla (n_-^2) = 0 ,
 $$
 then it can  be shown that $n_- = 0$ a.e. in $\Omega$, and therefore $n \geq 0$ a.e. in $\Omega$. 
 
 
 Now, we will  show the following pointwise bounds, not only on $z$ but also on $z/\lambda$:
\begin{equation} \label{lema est pontual z}
\alpha \leq \frac{z}{\lambda} \leq \norm{z^{i-1}}_{L^\infty} \quad \hbox{a.e. in $\Omega$}.
\end{equation}
 To prove \eqref{lema est pontual z},  the calculations are similar to the ones presented in \cite[Theorem 3.1]{Guillen-Gonzalez2023-gy}. The only difference is given by the convective term $u \cdot \nabla z$. However,  rewriting \eqref{3lamb} in terms of $\tilde{z} \coloneqq z/\lambda - \norm{z^{i-1}}_{L^\infty}$ and testing by $\tilde{z}_+$ where $\tilde{z}_+ \coloneqq  \max\left\{0,\tilde{z}\right\}$, the convective term becomes 
$$
\lambda \int (u \cdot \nabla z ) \tilde{z}_+ 
 = \lambda \int (u \cdot \nabla \tilde{z} ) \tilde{z}_+ 
= - \frac{\lambda}{2} \int u \cdot \nabla (\tilde{z}_+^2) = 0.
$$
 Similarly, when dealing with the lower bound, the equation is rewritten in terms of $\tilde{z} \coloneqq z/\lambda - \alpha$ and testing by $ \tilde{z}_{-}$, the convective term also vanishes.

Note that,  with pointwise estimates \eqref{lema est pontual z}  for $z/\lambda$ we also gain a  pointwise estimate for $z$ in $L^\infty$ (independent of $\lambda$).

To estimate $z$ in $ H^2$ we use  the energy structure of the $z$ equation. We first multiply \eqref{3lamb} by ${\lambda T_\alpha(z)}/{z}$, recalling that $z>0$, and arrive at
 \begin{align*}
     &\dfrac{T_\alpha(z)}{k } 
     -  T_\alpha(z) \dfrac{\Delta z}{z} 
     - \lambda \frac{\abs{\nabla z}^2}{z}  \\
     &= -\dfrac{s}{2 }G_0^m(n) T_\alpha(z) 
     + \lambda \dfrac{s\alpha^2}{ 2}G^m_0(n) \dfrac{1}{z} 
     + \lambda u \cdot \nabla z \dfrac{T_\alpha(z)}{z}
     + \lambda \dfrac{z^{i-1}}{k} \dfrac{ T_\alpha(z)}{z} .
 \end{align*}
By testing by $- \Delta z$, and integrating by parts, we have
 \begin{align} \label{teste energia existencia}
     & -\dfrac{1}{k } \int T_\alpha(z) \Delta z 
     +  \lambda  \int \frac{\abs{\nabla z}^2}{z} \Delta z
     + \int T_\alpha(z) \dfrac{\abs{\Delta z}^2}{z} 
     \nonumber \\ 
     &= -\dfrac{s}{2 }\int T_\alpha(z) \nabla G^m_0(n) \cdot  \nabla z 
     -\frac{s}{2}\int  G^m_0(n) (T_\alpha(z))' \abs{\nabla z}^2
     - \lambda \frac{ s \alpha^2}{2}\int  \dfrac{ G^m_0(n)}{ z^2} \abs{\nabla z}^2 
     \nonumber \\
     &+ \lambda \frac{ s \alpha^2}{2} \int \frac{1}{z} \nabla G^m_0(n) \cdot \nabla z
      - \lambda \int  u \cdot \nabla z\dfrac{T_\alpha(z)}{z} \Delta z 
     - \lambda\int \dfrac{z^{i-1}}{k} \dfrac{ T_\alpha(z)}{z} \Delta z.
     \end{align}
For the LHS of \eqref{teste energia existencia}, using the fact that for $T_\alpha(z)/ z \geq 1$ and integrating by parts
\begin{equation}
\begin{aligned}
    &-\dfrac{1}{k } \int T_\alpha(z) \Delta z 
    +  \lambda  \int \frac{\abs{\nabla z}^2}{z}\Delta z 
    + \int T_\alpha(z) \dfrac{\abs{\Delta z}^2}{z} \\ &
    \ge \dfrac{1}{k } \int T'_\alpha(z) \abs{\nabla z}^2  
    +  \lambda  \int \frac{\abs{\nabla z}^2}{z}\Delta z 
    + \int  \abs{\Delta z}^2 \\
    &\geq   \lambda  \int \frac{\abs{\nabla z}^2}{z}\Delta z 
    + \lambda \int \abs{\Delta z}^2 
    + (1-\lambda)\int \abs{\Delta z}^2 .
\end{aligned}
\end{equation}
Then by Lemma \ref{lema d2 z},
\begin{align*}
    &-\dfrac{1}{k } \int T_\alpha(z) \Delta z 
    +  \lambda  \int \frac{\abs{\nabla z}^2}{z}\Delta z 
    + \int T_\alpha(z) \dfrac{\abs{\Delta z}^2}{z} \\     
    & \geq C_3 \int \abs{\Delta z}^2
    - C_2 \lambda \int \abs{\nabla z}^2 
    + C_1 \lambda  \int \frac{\abs{\nabla z }^4}{z^2} .
\end{align*} 

Now, we deal with the RHS of \eqref{teste energia existencia}. 
Note that using the already obtained pointwise estimates \eqref{lema est pontual z} gives us
\begin{equation} \label{estimativa Tz/z}
     \lambda \frac{T_\alpha(z)}{z}  
 = \frac{T_\alpha(z)}{z/\lambda} \leq \frac{T_\alpha(z)}{\alpha} 
 \leq \frac{\norm{z^{i-1}}_{L^\infty}}{\alpha}.
\end{equation}
Therefore,  by first applying Hölder inequality and using \eqref{estimativa Tz/z} yields
\begin{equation}
     \lambda \int  u \cdot \nabla z \dfrac{T_\alpha(z)}{z} \Delta z 
\leq C \norm{u}_{H^1} \norm{\nabla z}_{L^4} \norm{\Delta z}_{L^2}.
\end{equation}
Now by applying Young's inequality twice, and multiplying and dividing the gradient term by $z^2/\lambda^2$, we have
\begin{align*}
\norm{u}_{H^1} \norm{\nabla z}_{L^4} \norm{\Delta z}_{L^2}  \leq C(\epsilon,\norm{z^{i-1}}_{L^\infty}) \norm{u}_{H^1}^4 
+ \lambda^2 \norm{z^{i-1}}_{L^\infty}^2 \epsilon \int \frac{\abs{\nabla z}^4}{z^2}  
+ \epsilon \norm{\Delta z}_{L^2}^2,
 \end{align*}
hence
 \begin{align*}
     \lambda \int  u \cdot \nabla z \dfrac{T_\alpha(z)}{z} \Delta z  
     \leq C(\epsilon,\norm{z^{i-1}}_{L^\infty},\norm{u}_{H^1}) 
     + \lambda \epsilon \int \frac{\abs{\nabla z}^4}{z^2} +  \epsilon \norm{\Delta z}_{L^2}^2.
 \end{align*}
 It also holds true that
 \begin{align*}
     \lambda\int \dfrac{z^{i-1}}{k} \dfrac{ T_\alpha(z)}{z} \Delta z  
     \leq C(\epsilon,\norm{z^{i-1}}_{L^\infty},k) 
     + \epsilon \norm{\Delta z}_{L^2}^2.
 \end{align*}
Moreover, using \eqref{eq Gm Tm} one has $\nabla G^m_0 (n)=T_0^m(n) n^{s-2} \nabla n =(2/s) T_0^m(n) n^{s/2-1} \nabla (n^{s/2})$, hence 
 \begin{align*}
     \frac{\lambda s\alpha^2}{2} \int \frac{1}{z} \nabla G^m_0( n) \cdot \nabla z 
     &= \lambda \alpha^2\int \frac{n^{s/2-1}}{z}T_0^m(n) \nabla z \cdot \nabla n^{s/2} 
     \\ &\leq  \frac{\lambda \alpha^3}{2} \int \frac{n^{s-2}}{z^2} T_0^m(n)^2\abs{\nabla z}^2 
     +  \frac{\lambda \alpha}{2}  \int \abs{\nabla n^{s/2}}^2.
 \end{align*}
By Lemma \ref{Lema dif T G}, one has $n^{s-2} T_0^m(n)^2\le s\, G_0^m(n)$, so
 \begin{align} \label{estimativa rhs}
     \frac{\lambda s \alpha^2}{2} \int \frac{1}{z} \nabla G^m_0( n) \cdot \nabla z 
     &\leq  \frac{\lambda s \alpha^3}{2} \int \frac{G^m_0(n)}{z^2}\abs{\nabla z}^2 
     +  \frac{\lambda \alpha}{2} \int  \abs{\nabla n^{s/2}}^2.
 \end{align}
  Then by setting $\alpha < 1$ the first term will be absorbed in \eqref{teste energia existencia}. We also note that $\nabla n^{s/2}$ might be singular for $s<2$. For an adaptation of this argument on the case $s<2$, the use of a translation on $n$ is necessary. This will be made more accurate in Lemma \ref{en s<2} below.

Altogether, we arrive at the inequality
\begin{align} \label{energia z existencia}
    &C_3 \int \abs{\Delta z}^2 
    - C_2 \lambda \int \abs{\nabla z}^2 
    +( C_1- \epsilon) \lambda \int \frac{\abs{\nabla z }^4}{z^2} 
    + \frac{s}{2}\int  G_0^m(n) (T_\alpha(z))' \abs{\nabla z}^2 \nonumber \\ 
    &\leq \frac{\lambda \alpha}{2} \int \abs{\nabla n^{s/2}}^2
   -\dfrac{s}{2 }\int \nabla G_0^m(n) \cdot T_\alpha(z) \nabla z 
   +C(\epsilon,\norm{u}_{H^1},\norm{z^{i-1}}_{L^\infty})  .
\end{align}

On the other hand, we multiply \eqref{2lamb} by $n^{s-1}$ and use \eqref{eq Gm Tm} to obtain
\begin{align*}
    \int \frac{n^s}{k}  
    +  \dfrac{4(s-1)}{s^2}\int \abs{\nabla n^{s/2}}^2 
    &= (s-1)\int T^m_0(n) n^{s-2}2 T_\alpha(z) \nabla z \cdot \nabla n
    + \frac{\lambda}{k}\int n^{i-1} n^{s-1} \\
    &\leq (s-1)\int 2 T_\alpha(z) \nabla z \cdot \nabla G^m_0(n) 
    + C(k,\norm{n^{i-1}}_{L^s}) 
    + \frac{1}{2k} \int n^s,
\end{align*}
and then
\begin{align} \label{en existencia n}
    \frac{s}{8(s-1)} \int \frac{n^s}{k} 
    +  \dfrac{1}{2s}\int \abs{\nabla n^{s/2}}^2 
    \leq \frac{s}{2} \int T_\alpha(z) \nabla z \cdot \nabla G^m_0(n) 
    + C(k,\norm{n^{i-1}}_{L^s}).
\end{align}
Summing up \eqref{energia z existencia} and \eqref{en existencia n}, and taking $\epsilon$ small enough, we get 
\begin{align*}
     &\frac{s}{8(s-1)} \int \frac{n^s}{k} 
     +  \dfrac{1}{s}\int \abs{\nabla n^{s/2}}^2
     + C_3 \int \abs{\Delta z}^2 
     + C_4 \lambda  \int \frac{\abs{\nabla z }^4}{z^2} 
     + \frac{s}{2}\int  G_0^m(n) (T_\alpha(z))' \abs{\nabla z}^2 
     \\ &\leq  \frac{\alpha}{2} \int \abs{\nabla n^{s/2}}^2 
     + C_2 \lambda \int \abs{\nabla z}^2 + C.
\end{align*}
Thus, if we set  $\alpha < 2/s$, we may absorb the  first term of the RHS. As for the second one, just note that
$$
 C_2 \lambda \int \abs{\nabla z}^2 \leq \lambda \epsilon \int \frac{\abs{\nabla z }^4}{z^2} + C(\epsilon,\norm{z^{i-1}}_{L^\infty}).
$$
Therefore $\Delta z $ is bounded in $L^2$ independently of $\lambda$, hence  $z$ is bounded in $H^2$ independently of $\lambda$ through the Poisson-Neumann regularity.

Now to check that $n$ is bounded in $ H^2$ (independently of $\lambda$), we start by testing \eqref{2lamb} by $n$, getting that $n $ is bounded independently of $\lambda$ in $ H^1$. Indeed,
\begin{align*}
        \dfrac{1}{k}\norm{n}^2_{L^2} + \norm{\nabla n}^2_{L^2} &= \int T_0^m(n) \nabla(z^2) \cdot \nabla n +  \dfrac{\lambda}{k} \int n^{i-1} n \\
   &\leq \epsilon \left( \dfrac{1}{k}\norm{n}^2_{L^2} + \norm{\nabla n}^2_{L^2}\right) + C(\epsilon,m,k,\norm{z}_{H^1},\norm{z^{i-1}}_{L^\infty},\norm{n^{i-1}}_{L^2}).
\end{align*}

On the other hand, since $z \in H^2$ we can use the identity 
\begin{align}
   \nabla \cdot (T^m_0(n) T_\alpha(z) \nabla z) =  \nabla n \cdot \nabla z (T^m_0)'(n) T_\alpha(z) + T^m_0(n) (T_\alpha(z))' \abs{\nabla z}^2 + T^m_0(n)T_\alpha(z) \Delta z .
\end{align}
Then $ \nabla \cdot (T^m_0(n) T_\alpha(z) \nabla z)$ is bounded  in  $L^2$ 
 and we may then test \eqref{2lamb} by $-\Delta n$. We observe that by interpolation 
\begin{align*}
    \int u \cdot \nabla n \, \Delta n &= -\int  (\nabla u  \, \nabla n) \cdot \nabla n \leq \int \abs{\nabla n}^2 \abs{\nabla u} \\
    & \leq C(\epsilon,  \norm{u}_{H^1}, \norm{n}_{H^1} ) + \epsilon \norm{\Delta n }_{L^2}^2 ,
\end{align*}
and therefore 
\begin{align*}
  &  \dfrac{1}{k} \int \abs{\nabla n}^2 + \int \abs{\Delta n}^2 = -2 \int \nabla \cdot (T^m_0(n) T_\alpha(z) \nabla z) \Delta n - \dfrac{\lambda}{k} \int n^{i-1} \Delta n +  \int u
 \cdot \nabla n \Delta n  \\
 &\leq C_1 ( \epsilon,m,k,\norm{z}_{H^2},\norm{z^{i-1}}_{L^\infty}, \norm{n^{i-1}}_{L^2}, \norm{u}_{H^1})  + 3\epsilon \int \abs{\Delta n}^2 .  \end{align*}
%
By taking $\epsilon$ small enough it follows that $n$ is bounded in $H^2$ (and in particular in $W^{1,4}$) independently of $\lambda$.

\

{\sl Step 3. Conclusion.}

 In summary, the fixed-point gives a triplet $(u^i,n^i,z^i) \in V \times (H^2)^2$, with an associated pressure $P^i \in H^1$, solution of the discrete scheme \eqref{Sist Discretizado}. 
 
 \end{proof}


\section{First uniform \texorpdfstring{$(m,k)$}{(m,k)}-estimates and energy inequalities} \label{secao 4}

By using the pointwise estimates made in previous section, estimate \eqref{lema est pontual z} (for $\lambda = 1$) readily implies that 
\begin{equation} \label{lema est pontual z-bis}
n^i\ge 0,\quad 
\alpha \leq z^i \leq \norm{z^{i-1}}_{L^\infty} 
\quad \text{a.e. in } \Omega.
\end{equation}
This guarantees the lower bound for $n^i$ and $z^i$ and also provides a $L^\infty$ bound for $z^i$. 
Moreover, from now on, the lower truncations given by \eqref{Talpha},  \eqref{T0m}, and \eqref{G0m} in the time discrete scheme \eqref{Sist Discretizado} may be simplified by making the following changes:
$$T_\alpha(z^i) = z^i,\quad T^m_0(n^i) = T^m(n^i) \quad \hbox{and} \quad G^m_0(n^i) = G^m(n^i),$$
where
$$     T^m(n)=\left\{ \ \begin{array}{cl}
n & \text { if } n \leq m, \\
m & \text { if } n \geq m ,
\end{array}\right. 
\quad 
G^m (n) = \left\{ \ \begin{array}{ll}
 \frac{1}{s}n^{s} &\text{ if } n\leq m, \\
 \frac{m}{s-1} n^{s-1} - \frac{m^s}{s(s-1)} & \text{ if } n\geq  m.
 \end{array}\right. 
 $$
 Specifically, time scheme  \eqref{Sist Discretizado} reduces to find $( n^i,z^i,u^i,P^i)\in (H^2)^2 \times(H^2\times V) \times H^1$ satisfying pointwise bounds \eqref{lema est pontual z-bis} and problem: 
 \begin{equation}  
  \label{Sist Discretizado-bis} 
 \left\{ \ \begin{aligned}  
&\delta_t n^i+u^i \cdot \nabla n^i  
-\Delta n^i
+\nabla \cdot(T^m(n^i)  2\, z^i \nabla z^i)=0, \\
&\delta_t z^i+u^{i} \cdot \nabla z^i 
-\Delta z^i
- \frac{\abs{\nabla z^i}^2}{z^i}
=-\dfrac{s}{2} G^m(n^i)\left( z^i - \dfrac{\alpha^2} {z^i}  \right) 
,  \\
&\delta_t  u^i
+u^{i-1} \cdot \nabla u^i 
-\Delta u^i
+\nabla P^i
 =
T^m(n^{i}) \nabla \Phi,  \\
 &\nabla \cdot u^i =0, 
\\ 
  &\left.\partial_\eta n^i \right|_{\Gamma}=\left.\partial_\eta z^i \right|_{\Gamma}
  = \left. u^i \right|_{\Gamma}= 0.
\end{aligned}\right.
\end{equation}

The next step is to obtain $(m,k)$-independent estimates.
The first result is obtained directly through testing the equations of the discrete scheme \eqref{Sist Discretizado-bis}, and reads as follows.

\begin{lema} \label{lema uniform estimat}
    Let $(n^i,z^i,u^i)$ be any solution of \eqref{Sist Discretizado-bis}. Then the following estimates hold
\begin{itemize}
    \item[i)] $\int n^i d x=\int n^0 d x$,\quad  for all $i \in \mathbb{N}$;
    \item[ii)] $\left\|z^i\right\|_{L^2}^2+\sum_{j=1}^i\left\|z^j-z^{j-1}\right\|_{L^2}^2 \leq\left\|z^0\right\|_{L^2}^2$ \quad for all $i \in \mathbb{N}$;
    \item[iii)] $k \sum_{j=1}^i\left\|\nabla z^j\right\|_{L^2}^2 \leq \frac{1}{4 \alpha^2}\left\|c^0+\alpha^2\right\|_{L^2}^2$ \quad for all $i \in \mathbb{N}$.
\end{itemize}
\end{lema}

\begin{proof}
All items will follow exactly from  \cite[Lemma 3.2]{Guillen-Gonzalez2023-gy}.
The first item is obtained by integrating $n^i$ equation, in which the convective term vanishes. The second and third items are obtained by testing $z^i$ equation, by  $z^i$ and $k(z^i)^3$ respectively, 
and, in both cases, the convective term also vanishes. In the latter case, it holds the intermediate inequality 
$k \sum_{j=1}^i\left\|\nabla (z^j)^2\right\|_{L^2}^2 \leq \left\| (z^0)^2 \right\|_{L^2}^2$, hence one arrives at  the third item using that $z^i\ge \alpha$.
\end{proof}

We now make use of the structure of the equations, and the cancellation effects between  chemotactic attraction  and  consumption terms, which is formally achieved by testing $n^i$ equation by $(n^i)^{s-1}$ and $z^i$ equation by $-\Delta z^i$ and balancing out the results. On the other hand,  we  use Lemma \ref{lema GNS produto} to bound the source term of the $u^i$ system. 

When the diffusion term $-\Delta n^i$ is tested by $(n^i)^{s-1}$, a term of the form $\nabla (n^i)^{s/2}$ appears, which may be singular in the case $s<2$. For this reason, the energy inequality is obtained separately for $s \in (1,2)$ and $s \geq 2$, avoiding this singular term $\nabla (n^i)^{s/2}$ in the case $s \in (1,2)$. 

\begin{lema} \label{en s<2}(Energy inequality for $s \in (1,2)$).
     Let $(n^i,z^i,u^i)$ be a solution of \eqref{Sist Discretizado-bis} with $s \in (1,2)$. Then, if  $\alpha$ is small enough,  there exist constants $C_1,C_2 ,C_3,C_4>0$ such that, 
      \begin{align} 
   &\frac{1}{2}  \delta_t \left[ \frac{1}{2(s-1)}\norm{n^i}_{L^s}^s + \norm{\nabla z^i}^2_{L^2} + C_1 \norm{u^i}_{L^2}^2 \right] 
    + \dfrac{1}{2k} \left[ \norm{\nabla z^i - \nabla z^{i-1}}_{L^2}^2 + C_1 \norm{u^i - u ^{i-1}}_{L^2}^2 \right] \nonumber \\ 
    & +  C_2 \norm{\nabla u^i}_{L^2}^2 +C_3 \left( \norm{D^2 z^i}_{L^2}^2 
    + \int \dfrac{\abs{\nabla z^i}^4}{(z^i)^2} \right)
    + \dfrac{s}{4} \int  G^m(n^i) \abs{\nabla z^i}^2 \leq   C_4.
\end{align}
\end{lema}

\begin{proof}
    We start by testing $z^i$ equation by $-\Delta z^i$, obtaining
\begin{align*}
   & \int \frac{z^{i-1}-z^i}{k}\Delta z^i 
   -\int  u^{i} \cdot \nabla z^i \Delta z^i 
   + \norm{\Delta z^i}_{L^2}^2 
   + \int \frac{\abs{\nabla z^i}^2}{z^i} \Delta z^i \\ 
   &=
    \dfrac{s}{2}\int \left( z^i - \dfrac{\alpha^2}{z^i}\right)G^m(n^i) \Delta z^i  
  \\
 & = -\frac{s}{2} \int  \left( z^i - \dfrac{\alpha^2}{z^i}\right) \nabla G^m(n^i) \cdot \nabla z^i
   - \dfrac{s}{2}\int \left(1+ \dfrac{\alpha^2}{(z^i)^2} \right) G^m(n^i) \abs{\nabla z^i}^2.
\end{align*}
Then, making use of Lemma \ref{lema d2 z} to bound the diffusive terms, bounding the convective term as
\begin{align*}
    \int  u^{i} \cdot \nabla z^i \Delta z^i 
    &= - \int  (\nabla u^i \, \nabla z^i) \cdot \nabla z^i  
    \leq \int \abs{\nabla z^i}^2 \abs{\nabla u^i} 
    \\&\leq \epsilon \int \frac{\abs{\nabla z^i}^4}{(z^i)^2} 
    + C_ \epsilon \int \abs{\nabla u^i}^2,
\end{align*}
one has, taking $\epsilon$ small enough and using \eqref{eq Gm Tm},
\begin{align}  \label{en z}
    &\delta_t \frac{1}{2}  \norm{\nabla z^i}^2_{L^2}
    + \dfrac{1}{2k} \norm{\nabla z^i 
    - \nabla z^{i-1}}_{L^2}^2 
    +  \frac12 C_1 \left( \norm{D^2 z^i}_{L^2}^2   + \int \dfrac{\abs{\nabla z^i}^4}{(z^i)^2} \right)
     + \dfrac{s}{2} \int   G^m(n^i) \abs{\nabla z^i}^2 
     \nonumber
    \\  
  &\leq  
  - \frac{s}{2} \int  z^i \left( 1 - \dfrac{\alpha^2}{(z^i)^2}\right) T^m(n^i) (n^i)^{s-2} \nabla n^i  \cdot \nabla z^i 
     +C_2\int \abs{\nabla z^i}^2
     + C_3 \int \abs{\nabla u^i}^2.
     \end{align}
     
 Due to the presence of the term  $\int \abs{\nabla u^i}^2$ in the RHS of \eqref{en z}, we look at the usual $u^i$ energy equation, testing the $u^i$ system by $u^i$ (accounting that $\int (u^{i-1} \cdot \nabla) u^i  \cdot u^i = 0$),
\begin{align}\label{en u}
& \delta_t \frac{1}{2}\norm{u^i}^2_{L^2} + \frac{1}{2k} \norm{u^i - u ^{i-1}}^2_{L^2} + \norm{\nabla u^i}^2_{L^2}   
= \int T^m(n^{i}) \, u^i \cdot \nabla \Phi
\le C_4 \int n^{i} \abs{u^i}.
\end{align}

Finally, we can then add up $2C_3$ times $\eqref{en u}$ and $\eqref{en z}$, 
and get

\begin{align} \label{en u2}
    &\delta_t\frac{1}{2}  \left[ \norm{\nabla z^i}^2_{L^2} +2C_3 \norm{u^i}_{L^2}^2 \right] 
    + \dfrac{1}{2k} \left[ \norm{\nabla z^i - \nabla z^{i-1}}_{L^2}^2 + 2C_3 \norm{u^i - u ^{i-1}}_{L^2}^2\right] 
    \nonumber \\ &
    +  C_3 \norm{\nabla u^i}_{L^2}^2 
    + \frac12 C_1 \left( \norm{D^2 z^i}_{L^2}^2 + \int \dfrac{\abs{\nabla z^i}^4}{(z^i)^2} \right)
    + \dfrac{s}{2} \int  G^m(n^i) \abs{\nabla z^i}^2 
     \nonumber \\ &
     \leq  
- \frac{s}{2} \int  z^i \left( 1 - \dfrac{\alpha^2}{(z^i)^2}\right) T^m(n^i) (n^i)^{s-2} \nabla n^i  \cdot \nabla z^i 
      +C_2\int \abs{\nabla z^i}^2
     + C_5 \int n^{i} \abs{u^i}.
\end{align}
Now, when testing $n^i$ equation by $(n^i)^{s-1}$
     the Laplacian term  becomes $\nabla(n^i)^{s/2}$ which is singular for $s<2$, and for that reason we translate the test function by $1/j$ and then pass the limit as $j\to +\infty$.
     Thus, testing $n^i$ equation by $(n^i+1/j)^{s-1}/{(s-1)}$, and using Lemma \ref{lema derivada discreta} there follows
  \begin{align*}
    &\frac{1}{s(s-1)} \delta_t \int (n^i +1/j)^s +  \dfrac{4}{s^2}\norm {\nabla(n^i + 1/j)^{s/2}}_{L^2}^2 + \frac{1}{2k} \int (b^i+1/j)^{s-2} (n^i - n^{i-1})^2  \\ &
    = \int T^m(n^i) (n^i+1/j)^{s-2} \nabla (z^i)^2 \cdot  \nabla n^i  .
    \end{align*}
    Since $(b^i+1/j)^{s-2} \geq 0$, we drop this term and arrive at
 \begin{align} \label{en n <2}
    \frac{1}{s(s-1)} \delta_t \int (n^i +1/j)^s 
    +  \dfrac{4}{s^2}\norm {\nabla(n^i + 1/j)^{s/2}}_{L^2}^2 
    \le \int T^m(n^i) (n^i+1/j)^{s-2} \nabla (z^i)^2 \cdot  \nabla n^i  .
    \end{align}
    So adding $s/4$ times $\eqref{en n <2}$ to $\eqref{en u2}$  we obtain
\begin{align*}
   &\frac{1}{2} \delta_t \left[ \frac{1}{2(s-1)}\norm{(n^i+1/j)}_{L^s}^s 
   + \norm{\nabla z^i}^2_{L^2} 
   + 2C_3 \norm{u^i}_{L^2}^2 \right] 
    + \dfrac{1}{2k} \left[ \norm{\nabla z^i - \nabla z^{i-1}}_{L^2}^2 
    + 2C_3 \norm{u^i - u ^{i-1}}_{L^2}^2 \right] 
    \\     &
    + \dfrac{1}{s}\norm {\nabla(n^i + 1/j)^{s/2}}_{L^2}^2
    +  C_3 \norm{\nabla u^i}_{L^2}^2 
    + \frac12 C_1 \big( \norm{D^2 z^i}_{L^2}^2 
    + \int \dfrac{\abs{\nabla z^i}^4}{(z^i)^2} \big)
    + \dfrac{s}{2} \int  G^m(n^i) \abs{\nabla z^i}^2 \nonumber 
    \\ & 
     \nonumber 
    \leq 
    \frac{s}{2} \int  \dfrac{\alpha^2}{z^i} T^m(n^i) (n^i)^{s-2} \nabla n^i  \cdot \nabla z^i 
      +C_2\int \abs{\nabla z^i}^2
    + C_5 \int n^{i} \abs{u^i}\\ 
     & - \frac{s}{4}\int T^m(n^i) \nabla (z^i)^2 \cdot  \nabla n^i \left( (n^i)^{s-2} - (n^i+1/j)^{s-2} \right).
\end{align*}

For the first term on the RHS we note that
\begin{equation}
     \nabla  n^i = \nabla (n^i +1/j) = \nabla  ((n^i +1/j)^{s/2})^{2/s}  
     = \frac{2}{s}(n^i +1/j)^{1-s/2} \nabla(n^i + 1/j)^{s / 2},
\end{equation}
hence this first term remains 
\begin{equation*} 
    \begin{aligned}
   & 
    \int   \dfrac{\alpha^2}{z^i} T^m(n^i) (n^i)^{s-2} (n^i +1/j)^{1-s/2} \nabla(n^i + 1/j)^{s / 2} \cdot \nabla z^i   
   \\
    &\leq  C(\epsilon)\alpha^2 \int (n^i)^{2s-4} T^m(n^i)^2 (n^i +1/j)^{2-s} \abs{\nabla z^i}^2 
     +  \epsilon  \int \abs{\nabla (n^i+1/j)^{s/2}}^2 \\
     &\leq  C(\epsilon)\alpha^2 s \int G^m(n^i)\abs{\nabla z^i}^2 (n^i)^{s-2} (n^i +1/j)^{2-s} 
     +  \epsilon  \int \abs{\nabla (n^i+1/j)^{s/2}}^2.
    \end{aligned}
\end{equation*}
 In the last line, we have used Lemma \ref{Lema dif T G}. First, we choose $\epsilon$ small enough to absorb the second term and then set a small $\alpha$ accordingly, depending on $\epsilon$. We also note that even with the presence of the singular term $(n^i)^{s-2}$ the whole term is still not singular because close to zero $G^m(n^i)$ is of the order of $(n^i)^s$ making the product $(n^i)^{s-2} G^m(n^i)$ well defined. Overall, for $\alpha$ small enough we get
\begin{equation} \label{estimativa rhs-bis}
    \dfrac{s }{2} \int \frac{\alpha^2}{z^i} T^m(n^i) (n^i)^{s-2} \nabla n^i  \cdot \nabla z^i   \leq \frac{s}{4} \int G^m(n^i)\abs{\nabla z^i}^2 (n^i)^{s-2} (n^i +1/j)^{2-s} + \frac{1}{4s} \int \abs{\nabla (n^i+1/j)^{s/2}}^2.
\end{equation}

 For the source term of the $u^i$ system, 
  applying Lemma \ref{lema GNS produto} with $f = n^{i}$ and $g = u^i$ and choosing adequate $\epsilon_1$ and $\epsilon_2$ we have that
$$
 C_5 \int_{\Omega} n^{i} \abs{u^i} 
 \leq  C_5 \int_{\Omega} (n^{i} +1/j)\abs{u^i} 
 \leq \frac{C_3}{2} \norm{\nabla u^i}^2_{L^2} + \frac{1}{4s}\norm{\nabla (n^i+1/j)^{s/2}}_{L^2}^{2} + C_6.
$$
We can also estimate 
\begin{equation}\label{est grad z L2}
    C_2\int \abs{\nabla z^i}^2 = C_2\int \frac{\abs{\nabla z^i}^2}{z^i}z^i 
   \leq \epsilon \int \frac{\abs{\nabla z^i}^4}{(z^i)^2} + C_ \epsilon. 
\end{equation}

Dropping the gradient term of $(n^i+1/j)^{s/2}$, avoiding any division by zero, we get that
 \begin{align*}
   &\frac{1}{2}  \delta_t \left[ \frac{1}{2(s-1)}\norm{n^i+1/j}_{L^s}^s + \norm{\nabla z^i}^2_{L^2} + 2C_3 \norm{u^i}_{L^2}^2 \right] 
    + \dfrac{1}{2k} \left[ \norm{\nabla z^i - \nabla z^{i-1}}_{L^2}^2 
    + 2C_3 \norm{u^i - u ^{i-1}}_{L^2}^2 \right] \\ 
    & + \frac{C_3}{2}\norm{\nabla u^i}_{L^2}^2 
    +\frac14 C_1 \big( \norm{D^2 z^i}_{L^2}^2 
    + \int \dfrac{\abs{\nabla z^i}^4}{(z^i)^2} \big)
    + \dfrac{s}{2} \int  G^m(n^i) \abs{\nabla z^i}^2 \Bigl(1 - \dfrac{(n^i)^{s-2} (n^i +1/j)^{2-s} }{2} \Bigr)
    \nonumber \\ &\leq  C_6
    - \frac{s}{4} \int T^m(n^i)  \nabla (z^i)^2 \cdot \nabla n^i \left( (n^i)^{s-2} - (n^i+1/j)^{s-2} \right) .
\end{align*}

Finally, we can pass to the limit as $j \to \infty$, via the Dominated Convergence Theorem, to finish the proof, remarking only that 
\begin{equation}
\begin{aligned}
    \int  G^m(n^i) \abs{\nabla z^i}^2 (n^i)^{s-2} (n^i +1/j)^{2-s} &\leq \dfrac{1}{s}\int \abs{\nabla z^i}^2 (n^i)^s(n^i)^{s-2} (n^i +1/j)^{2-s}  
    \\&\le  \dfrac{1}{s}\int \abs{\nabla z^i}^2 (n^i+1)^{s},
\end{aligned}
    \end{equation}
    which is a $L^1$ function, and then 
    \begin{equation}
         \dfrac{s}{2} \int  G^m(n^i) \abs{\nabla z^i}^2 (1 - \dfrac{(n^i)^{s-2} (n^i +1/j)^{2-s} }{2} ) \to \dfrac{s}{4} \int  G^m(n^i) \abs{\nabla z^i}^2.
    \end{equation}
    The remaining terms can be treated similarly.
\end{proof}

\begin{lema} \label{en s>2} (Energy inequality for $s \geq 2$).
       Let $(n^i,z^i,u^i)$ be a solution of \eqref{Sist Discretizado-bis} with $s \geq 2$. Then, if $\alpha$ is small enough,  there exist constants $C_1,C_2,C_3,C_4 >0$ such that 
\begin{align*} 
   &\delta_t\frac{1}{2}  \left[ \frac{1}{2(s-1)} \norm{n^i}_{L^s}^s + \norm{\nabla z^i}^2_{L^2} + C_1 \norm{u^i}_{L^2}^2\right] 
    + \dfrac{1}{2k} \left[ \norm{\nabla z^i - \nabla z^{i-1}}_{L^2}^2 +  C_1 \norm{u^i - u ^{i-1}}_{L^2}^2   \right] \nonumber 
    \\     &
    + \dfrac{1}{4s}\norm {\nabla(n^i)^{s/2}}_{L^2}^2
    +  C_2 \norm{\nabla u^i}_{L^2}^2 +C_3 \left( \norm{D^2 z^i}_{L^2}^2 
    + \int \dfrac{\abs{\nabla z^i}^4}{(z^i)^2} \right) 
    + \dfrac{s}{4} \int  G^m(n^i) \abs{\nabla z^i}^2
     \leq C_4.
\end{align*}
\end{lema}
\begin{proof}
  Up until \eqref{en u2} there is no distinction whether $s \in (1,2)$ or $ s \geq 2$. Now as  there is no singularity, testing $n^i$ equation by $(n^i)^{s-1}/{(s-1)}$ and using Lemma \ref{lema derivada discreta} yields

\begin{align} \label{en n +2}
    \frac{1}{s(s-1)} \delta_t \int (n^i)^s 
    +    \dfrac{4}{s^2}\norm {\nabla(n^i)^{s/2}}_{L^2}^2 
    \leq \int T^m(n^i) (n^i)^{s-2} \nabla (z^i)^2 \cdot  \nabla n^i.
\end{align}

Adding $s/4$ times $\eqref{en n +2}$ to $\eqref{en u2}$, we obtain
\begin{align*}
   &\delta_t\frac{1}{2}  \left[ \frac{1}{2(s-1)}\norm{n^i}_{L^s}^s + \norm{\nabla z^i}^2_{L^2} + 2C_3 \norm{u^i}_{L^2}^2 \right] 
    + \dfrac{1}{2k} \left[ \norm{\nabla z^i - \nabla z^{i-1}}_{L^2}^2 + 2C_3 \norm{u^i - u ^{i-1}}_{L^2}^2 \right] \\ 
    &+ \dfrac{1}{s}\norm {\nabla(n^i)^{s/2}}_{L^2}^2
    +  C_3 \norm{\nabla u^i}_{L^2}^2 +\frac12 C_1 \Bigl( \norm{D^2 z^i}_{L^2}^2 
    + \int \dfrac{\abs{\nabla z^i}^4}{(z^i)^2} \Bigr)
    + \dfrac{s}{2} \int  G^m(n^i) \abs{\nabla z^i}^2 \nonumber
    \\ &
     \leq 
     \frac{s}{2} \int  \dfrac{\alpha^2}{z^i} T^m(n^i) (n^i)^{s-2} \nabla n^i  \cdot \nabla z^i 
+C_2\int \abs{\nabla z^i}^2
     + C_5 \int n^{i} \abs{u^i}. 
\end{align*}

Following the same calculations  to get \eqref{estimativa rhs-bis} changing $(n^i+1/j)$ by $n^i$, we have, for $\alpha$ small enough, that
 \begin{align*} 
     \frac{s}{2} \int \frac{ \alpha^2}{z^i}  T^m(n^i) (n^i)^{s-2} \nabla n^i \cdot \nabla z^i 
     &\leq  \frac{s }{4 } \int G^m(n^i) \abs{\nabla z^i}^2 
     + \frac{1}{4s }  \int  \abs{\nabla (n^i)^{s/2}}^2.
 \end{align*}

    For the source term of $u^i$ system, applying Lemma \ref{lema GNS produto} with $f = n^{i}$ and $g = u^i$ and choosing adequate $\epsilon_1$ and $\epsilon_2$, we get
$$
 C_5 \int_{\Omega} n^{i} \abs{u^i}  \leq \frac{C_3}{2} \norm{\nabla u^i}^2_{L^2} + \frac{1}{4s}\norm{\nabla (n^i)^{s/2}}_{L^2}^{2} + C_6.
$$
By plugging \eqref{est grad z L2},   we obtain the result.
\end{proof}


\section{Energy \texorpdfstring{$(m,k)$}{(m,k)}-estimates and passage to the limit as \texorpdfstring{$(m,k) \to (\infty,0)$}{(m,k) -> (inf,0)}} \label{secao 5}

 With the $(m,k)$-estimates which will provide  Lemmas \ref{lema uniform estimat}-\ref{en s>2} at hand, we are now prepared to prove our main result, Theorem \ref{resultado principal}.  
We start rewriting the scheme \eqref{Sist Discretizado-bis} as a continuous in-time system, by 
considering the piecewise constant function $n_m^k$ and the locally linear and globally continuous function $\tilde{n}_m^k$ defined  by 
$$n_m^k(t, x)=n^i(x)
\quad \text { if } t \in I_i
$$
 and
$$
\tilde{n}_m^k(t, x)=n^i(x)+\frac{\left(t-t_i\right)}{k}\left(n^i(x)-n^{i-1}(x)\right) 
\quad \text { if } t \in I_i.
$$
Analogously, we define the functions $z_m^k, \tilde{z}_m^k, u_m^k$, $\tilde{u}_m^k$ and $P^k_m$.
We also define 
$$\hat{u}_m^k (t,x) = u^{i-1}(x) \quad \text {if  $t \in I_i$.}$$
In particular, one has  $\partial_t \tilde{n}_m^k = \delta_t n^i$ for $t \in I_i$.
Then, the discrete scheme \eqref{Sist Discretizado-bis} can be rewritten as

\begin{equation} \label{sistema edp}
\left \{ \begin{aligned}
    & \partial_t \tilde{n}_m^k + u_m^k \cdot \nabla n_m^k -\Delta n_m^k 
    + \nabla \cdot\left(T^m(n_m^k) \nabla(z_m^k)^2\right)=0, \\
& \partial_t \tilde{z}_m^k + u_m^k \cdot \nabla z_m^k -\Delta z_m^k
- \frac{\left|\nabla z_m^k\right|^2}{z_m^k}
=-\frac{s}{2} G^m(n_m^k) \left( z_m^k-\frac{\alpha^2}{z_m^k}   \right) ,  \\
& \partial_t  \tilde{u}_m^k +(\hat{u}_m^k \cdot \nabla) u_m^k - \Delta u_m^k +\nabla P_m^k =T^m(n_m^k) \nabla \Phi,  \\
 & \nabla \cdot u_m^k  =0, \\
 & \left.\partial_\eta \tilde{n}_m^k\right|_{\Gamma}=\left.\partial_\eta \tilde{z}_m^k \right|_{\Gamma}= \left. \tilde{u}_m^k\right|_{\Gamma}= 0 , \quad \tilde{n}(0) = n^0, \quad \tilde{z}(0) = z^0, \quad \tilde{u}(0) = u^0.
\end{aligned} \right.
\end{equation}

As a direct consequence of those definitions, from estimate \eqref{lema est pontual z-bis} and Lemma \ref{lema uniform estimat}, one has 
\begin{equation}\label{n-L1,z-Linf}
\int n^k_m = \int n^0\quad \text{and} \quad \norm{z^k_m}_{L^\infty} \leq \norm{z^0}_{L^\infty},
\end{equation}
and the following bound (independent of $m,k$) holds
\begin{equation} \label{z linf l2 s<2}
   \nabla z_m^k \text{  in } L^2(0, \infty;L^2).
\end{equation}

\subsection{First energy estimates} 
  We sum up the energy inequalities obtained in Lemmas \ref{en s<2} and \ref{en s>2}, from $j=1$ to $i$, for any $i$, and multiply it by $k$ to obtain
   \begin{align} \label{en dado inicial s<2}
\displaystyle
       &   \frac{1}{4(s-1)}\int\left(n^i\right)^s 
        + \frac{1}{2} \int \abs{\nabla z^i}^2 
        + \frac{C_1}{2}\int \abs{u^i}^2  
        + \dfrac{1}{2}  \sum_{j=1}^i \left[ \norm{\nabla z^j - \nabla z^{j-1}}_{L^2}^2 
        + C_1 \norm{u^j - u ^{j-1}}_{L^2}^2\right] 
         \nonumber \\ 
         \displaystyle
      &+ k\sum_{j=1}^i \left(C_2  \norm{\nabla u^j}^2_{L^2} 
       +  C_3 \left(\norm{D^2 z^j}_{L^2}^2 
       + \int \frac{\abs{\nabla z^j}^4}{(z^j)^2}\right) 
       +  \frac{s}{4} \int  G^m(n^j) \abs{\nabla z^j}^2 \right)  
       \nonumber  \\
        \displaystyle
        &\leq    
          \frac{1}{4(s-1)}\int (n^0)^s
        + \frac{1}{2} \int \abs{\nabla z^0}^2 
        +  \frac{C_1}{2} \int \abs{u^0}^2
        +  C_4 \, t_i  .
         \end{align}
         
         Note that, for each regularized initial data $(n^0_j,z^0_j,u^0_j)$ 
 satisfying  \eqref{reg n0},  \eqref{compat-0} and \eqref{approx-0},  the estimates obtained from the energy inequality only depend on 
 $\| n^0_j\|_{L^s}^s + \|\nabla z^0_j\|^2_{L^2} +  \|u^0_j\|_{L^2}^2$ which is bounded independently of $j$. Then we may conclude that the RHS of \eqref{en dado inicial s<2} is bounded at finite time, therefore  for any fixed $T>0$, the following bounds hold
\begin{equation} \label{grad z linf l2 s<2}
   \nabla z_m^k \text{  in } L^\infty(0,T;L^2) \cap L^2(0,T;H ^1) \cap L^4(0,T;L^4),
\end{equation}
\begin{equation} \label{u linfl2 s<2}
      u_m^k \text{  in }L^\infty(0,T;H)\cap L^2(0,T;V),
\end{equation}

\begin{equation} \label{n linf ls s<2}
     n_m^k \text{  in } L^\infty(0,T;L^s),
\end{equation}
\begin{equation} \label{grad u lim l2 l2}
  G^m(n_m^k)^{1/2} \nabla z_m^k  \text{  in }L^2(0,T;L^2).
\end{equation}
Note that by Lemma \ref{Lema dif T G} one has $T^m(n^i)^s \leq  s \, G^m(n^i)$, hence from \eqref{grad u lim l2 l2}
\begin{align} \label{lim tm s grad z l2l2} 
 T^m(n^i)^{s/2}  \nabla z^i \text{ is bounded in } L^2(0,T;L^2) .  
\end{align}
On the other hand, estimates \eqref{n-L1,z-Linf}  and \eqref{grad z linf l2 s<2}    imply that 
\begin{equation}\label{grad z lim}
   z_m^k \text{ is bounded in } L^\infty(0,T;H^1\cap L^\infty)\cap L^2(0,T;H^2).
 \end{equation}
By starting from \eqref{u linfl2 s<2}  and using  $3D$ interpolation, one also has  
\begin{align} \label{u l2l6 lim 10/3 10/3}
    u_m^k  \text{ is bounded in }
    L^{10/3}(0,T;L^{10/3}).
\end{align}
Lemma \ref{lema uniform estimat} and energy estimate \eqref{en dado inicial s<2} also  imply 
$$
\sum_{j=1}^i \norm{z^j - z ^{j-1}}_{H^1}^2 + \sum_{j=1}^i \norm{u^j - u ^{j-1}}_{L^2}^2 
\leq C \,T,\quad \forall\, i \geq 1,
$$
hence we may infer, 
%
\begin{equation} \label{conv cont pra por partes}
     \norm{z_m^k - \tilde{z}_m^k}_{L^2(0,T;H^1)}^2 + \norm{u_m^k - \tilde{u}_m^k}_{L^2(0,T;L^2)}^2 + \norm{\hat{u}_m^k-\tilde{u}_m^k}_{L^2\left(0, T ; L^2\right)}^2 
     \leq C \,T\, k .
\end{equation}

\subsection{Additional estimates for \texorpdfstring{$s\in(1,2)$}{s in (1,2)}.}
We next obtain an estimate for the difference $n_m^k-\tilde{n}_m^k$ and for $n^k_m$ and $\nabla n^k_m$. 
\begin{lema} \label{le:5.1}
   For $s \in (1,2)$, there exists $C>0$, independent of $m$ and $k$, such that
   \begin{equation} \label{conv n ntil}
    \norm{n_m^k-\tilde{n}_m^k}_{L^2\left(0, T ; L^s \right)}^2 \leq C\, k .
\end{equation}
 Moreover, the following estimates hold
  \begin{align} \label{lim n 5s 3 5s 3 s<2}
    n_m^k \text{  in } L^{5s/3}(0,T;L^{5s/3}),
\end{align}
    and
 \begin{equation} \label{lim grad n l2ls}
     \nabla n_m^k  \text{  in }  L^{5s/(s+3)}(0,T;L^{5s/(s+3)}).
\end{equation}
\end{lema}

\begin{proof}
Testing $n^i$ equation of \eqref{Sist Discretizado-bis} 
by $ (n^i+1)^{s-1}/(s-1)$, and using Lemma \ref{lema derivada discreta}, we have 
\begin{align*}
    &\frac{1}{s(s-1)} \delta_t \int (n^i +1)^s +  \dfrac{4}{s^2}\norm {\nabla(n^i + 1)^{s/2}}_{L^2}^2 + \frac{1}{2k} \int (b^i +1)^{s-2} (n^i - n^{i-1})^2  \\ &
    = \int T^m(n^i) (n^i+1)^{s-2} \nabla (z^i)^2 \cdot  \nabla n^i  ,
    \end{align*}
    where $b^i(x)$ is an intermediate function between $n^i(x)$ and $n^{i-1}(x)$.
Then by using  that $1-s/2 >0$ (recall that $s<2$), the equality
$$
(n^i +1)^{s/2-1} \nabla  n^i 
 = \frac{2}{s} \nabla(n^i + 1)^{s / 2},
$$
 and the inequality  $T^m(n^i) \leq n^i$, 
one has 
\begin{align*} 
    &\frac{1}{s(s-1)} \delta_t \int (n^i +1)^s +  \dfrac{4}{s^2}\norm {\nabla(n^i + 1)^{s/2}}_{L^2}^2 
    \nonumber\\ &
    \le 2 \int T^m(n^i)(n^i+1)^{s / 2-1} z^i \nabla z^i \cdot \nabla n^i(n^i+1)^{s / 2-1}  \nonumber\\
& = \frac{4}{s} \int \frac{T^m\left(n^i\right)^{1-s / 2}}{\left(n^i+1\right)^{1-s / 2}} T^m(n^i)^{s / 2} z^i \nabla z^i \cdot \nabla(n^i+1)^{s / 2} \nonumber  \\
& \leq C \left\|z^0\right\| ^2_{L^{\infty}}\int T^m(n^i)^s\left|\nabla z^i\right|^2 
 + \dfrac{2}{s^2}\norm {\nabla(n^i + 1)^{s/2}}_{L^2}^2.
    \end{align*}
Thus, we arrive at 
    \begin{equation} \label{en aux s<2}
\frac{1}{s(s-1)} \delta_t \int (n^i +1)^s +  \dfrac{2}{s^2}\norm {\nabla(n^i + 1)^{s/2}}_{L^2}^2 
\leq C \int T^m(n^i)^s\left|\nabla z^i\right|^2 .
\end{equation}

Now multiplying \eqref{en aux s<2} by $k$ and summing up, from $1$ to $i$, for any $i \in \mathbb{N}$, we obtain
$$
\begin{aligned}
&\frac{1}{s(s-1)} \int (n^i +1)^s +  \frac2{s^2} k \sum_{j=1}^i\norm {\nabla(n^j + 1)^{s/2}}_{L^2}^2 
 \\
&\leq C k \sum_{j=1}^i \int T^m(n^j)^s\left|\nabla z^j\right|^2  
 +\frac{1}{s(s-1)}\int (n^0+1)^s \\
 &\leq C_1(T)  + 
 C_2\left( \int (n^0)^s + |\Omega|\right),
\end{aligned}
$$
where estimate \eqref{lim tm s grad z l2l2} has been applied in the last inequality.  
Therefore, we infer 
\begin{equation} \label{(n+1)^s/2-weak}
 (n_m^k + 1)^{s/2} \text{ is bounded in } L^\infty(0,T;L^2) \cap L^2(0,T;H^1).
 \end{equation}
Note that \eqref{(n+1)^s/2-weak} implies, by interpolation, that
$$
    (n_m^k+1) ^{s/2} \text{ is bounded in } 
    L^{10/3}(0,T;L^{10/3}),
$$
hence \eqref{lim n 5s 3 5s 3 s<2} is proved. 

 On the other hand, we can write
$$
\nabla  n^k_m 
 = \frac{2}{s}(n^k_m +1)^{1-s/2} \nabla(n_m^k + 1)^{s / 2}.
$$
Since $s<2$, from \eqref{lim n 5s 3 5s 3 s<2} we  get that $(n^k_m +1)^{1-s/2}$ is bounded in 
$L^{10 s/ 3(2-s)}(0,T; L^{10 s/ 3(2-s)})$, which jointly with \eqref{(n+1)^s/2-weak} 
imply \eqref{lim grad n l2ls}.
  
Finally, using previous bounds and following line by line \cite[Lemma 4.2]{Guillen-Gonzalez2023-gy}, we can obtain \eqref{conv n ntil}.
\end{proof}

The next lemma provides uniform in time energy estimates.
\begin{lema} \label{energy-estim s<2}
For $ s \in(1,2)$, we have 
 \begin{equation}  \label{uniform in time energy estimates}
(n_m^k, \nabla z_m^k,u_m^k) \text{ is bounded in } L^\infty(0,\infty;L^s\times L^2\times H).
\end{equation}
\end{lema}
\begin{proof}
 By Lemma \ref{Lema dif T G} we have that $T^m(n^i)^s \leq sG^m(n^i)$, so we can add  $1/(8C)$ times \eqref{en aux s<2} to the energy inequality given in Lemma \ref{en s<2}, and obtain
the recursive inequality
\begin{equation} \label{recursive inequality}
\delta_t a_i + d_i \leq C_4,
\end{equation}
where
\begin{align} \label{Def Funcional Energia Discreto s<2}
    a_i :=  \frac{1}{4(s-1)} \norm{n^i}_{L^s}^s+ \frac{1}{8Cs(s-1)} \norm{n^i+1}_{L^s}^s +  \frac{1}{2}\norm{\nabla z^i}_{L^2}^2 + \frac{C_1}{2}\norm{u^i}_{L^2}^2 ,
\end{align}
and 
\begin{align} \label{Def funcional dispersivo s<2}
    d_i :=  \dfrac{1}{4Cs^2}\norm {\nabla(n^i+1)^{s/2}}_{L^2}^2
    +  C_2 \norm{\nabla u^i}_{L^2}^2 +C_3 \left( \norm{D^2 z^i}_{L^2}^2 
    + \int \frac{\abs{\nabla z^i}^4}{(z^i)^2} \right)
    + \dfrac{s}{8} \int  G^m(n^i) \abs{\nabla z^i}^2 .
\end{align}

Now we show that there exist $K_1, K_2>0$  such that 
\begin{equation} \label{est ai}
    a_i \leq K_1 d_i + K_2,  \quad \text{ for all } i\ge 1.
\end{equation}
We estimate each term of $a_i$ separately.  
For $n^i$, we note that by applying Poincaré's inequality in $(n^i+1)^{s/2}$ one has
     \begin{align} \label{est n+1a}
      \norm{n^i+1}_{L^s}^s =      \norm{(n^i+1)^{s/2}}_{L^2}^2 &\leq C\norm{\nabla (n^i+1)^{s/2}}_{L^2}^2 + C \norm{(n^i+1)^{s/2}}_{L^1}^2.
     \end{align}
For the second term on the RHS, as $ s < 2 $, by using H\"older inequality, we have that
    \begin{equation}\label{est n+1}
     \norm{(n^i+1)^{s/2}}_{L^1} ^2 \leq |\Omega|^{2-s}\norm{n^i+1}_{L^1}^s \leq C\bigl(\norm{n^0}_{L^1}\bigr),
       \end{equation}
       where in the last estimate  we have used Lemma \ref{lema uniform estimat}.
   %
       Therefore,
      \begin{align*}
        \frac{1}{4(s-1)} \norm{n^i}_{L^s}^s+ \frac{1}{8C(s-1)} \norm{n^i+1}_{L^s}^s &\leq C \norm{\nabla (n^i+1)^{s/2}}^2_{L^2} + K_2 \leq K_1 d_i + K_2.
       \end{align*}
  For $z^i $, we use its $L^\infty$ bound
    \begin{align*}
        \norm{\nabla z^i}^2_{L^2} &\leq \norm{z^i}_{L^\infty} \int \frac{\abs{\nabla z^i}^2}{z^i} \leq \frac{1}{2}  \int \frac{\abs{\nabla z^i}^4}{(z^i)^2} + K_2 \leq \frac{1}{2}d_i + K_2.
    \end{align*}
Finally,     for $u^i$ we just note that by Poincaré's inequality, there exists $C >0 $ such that
    \begin{equation}
        \norm{u^i}_{L^2}^2 \leq  C \norm{\nabla u^i}^2_{L^2}  \leq K_1 d_i.
    \end{equation}
Plugging the previous bounds, we find \eqref{est ai}.
Hence, the recursive inequality \eqref{recursive inequality} can be written as
\begin{equation} \label{global-dt-ineq}
\delta_t a_i  + C_6 a_i \leq C_4+C_5 . 
\end{equation}
Then, by applying Lemma \ref{Lema Gronwall Discreto} with constants $\lambda = C_6$ and  $C = C_4+C_5$ yields to 
\begin{align*}
    a_i \leq a_0 + C/C_6, \quad \forall\, i \ge 1.
\end{align*} 
By taking into account the expression of $a_i$, one has the uniform in time energy estimates \eqref{uniform in time energy estimates}. 
\end{proof}

Now, we look for an estimate for $\partial_t \tilde{n}^k_m$. 
Considering  the  chemotaxis flux, as $s<2$ we may write $T^m(n_m^k) \nabla(z_m^k)^2$ as
$$
T^m(n_m^k) \nabla(z_m^k)^2
=2\, T^m(n_m^k)^{1-s / 2} T^m(n_m^k)^{s / 2} z_m^k \nabla z_m^k .
$$
Then, using from \eqref{lim n 5s 3 5s 3 s<2} that  $(n^k_m )^{1-s/2}$ (and $T^m(n_m^k)^{1-s / 2}$) is bounded in 
$L^{10 s/ 3(2-s)}(0,T; L^{10 s/ 3(2-s)})$,
and using also estimates \eqref{n-L1,z-Linf} and  \eqref{lim tm s grad z l2l2} we conclude
\begin{equation} \label{tm grad z2 l2 ls}
    T^m(n_m^k) \nabla(z_m^k)^2 \text { is bounded in } L^{5s/(s+3)}(0,T;L^{5s/(s+3)}). 
\end{equation}
On the other hand, we rewrite the coupled term $u^k_m \cdot \nabla n^k_m=\nabla\cdot (u^k_m n^k_m ) $, and from \eqref{u l2l6 lim 10/3 10/3} and
 \eqref{lim n 5s 3 5s 3 s<2}  
\begin{equation} \label{un 10s10s}
    u^k_m n^k_m  \text { is bounded in }L^{10s/(3s+6)}(0,T;L^{10s/(3s+6)}).
\end{equation}
By applying the previous estimates in the equation $\eqref{sistema edp}_1$, the following 
 bound of the time derivative of $\tilde{n}^k_m$ holds 
\begin{equation} \label{d-t-n}
    \partial_t \tilde{n}^k_m  \text { is bounded in }  L^{10/9}(0,T;(W^{1,10})').
\end{equation}

\


   Next, we deduce an estimate for $\partial_t z^k_m$. Note that by \eqref{grad z linf l2 s<2} 
   and \eqref{lema est pontual z-bis},
    $$
    \frac{\abs{\nabla z_m^k}^2}{z_m^k}  \text { is bounded in } L^2(0,T;L^2).
    $$
    Moreover as $s\, G^m(n_m^k) \leq (n_m^k)^s$,
     which is bounded in $L^{5/3}(0,T;L^{5/3})$ owing to \eqref{lim n 5s 3 5s 3 s<2}, 
%
    then 
    $$
    G^m(n_m^k), \ \frac{G^m(n_m^k)}{z_m^k}  \text { are bounded in } L^{5/3}(0,T;L^{5/3}).
    $$
    For the convective term $ u_m^k \cdot \nabla z_m^k$, since $u^k_m $ and $\nabla z^k_m $ are bounded in $ L^{10/3}(0, T ; L^{10/3})$, it holds
\begin{equation} \label{uz reg}
    u_m^k \cdot \nabla z_m^k  \text { is bounded in } L^{5/3}(0,T;L^{5/3}).
\end{equation}
Finally, from   equation $\eqref{sistema edp}_2$ we may conclude that
\begin{equation} \label{zt lim l2l3/2}
    \partial_t \tilde{z}_m^k  \text { is bounded in } L^{5/3}(0,T;L^{5/3}).
\end{equation}


\subsection{Estimates for linear in time approximations in the case \texorpdfstring{$s\in(1,2)$}{s in (1,2)}.} \label{est lin por partes}
In order to obtain estimates for the global continuous  and local linear in-time sequence $\tilde{n}_m^k$ (and $\tilde{z}_m^k, \tilde{u}_m^k$), we use the triangular inequality
$$
\norm{\tilde{n}_m^k(t)} \leq \norm{n_m^k(t)} 
+\left\|n^{j-1}\right\|, \quad \forall \, t \in I_j .
$$
Since there is a dependence on the previous time step,  more regular initial data is required, as we have stated in \eqref{reg n0}-
\eqref{reg u0}. 
Then,  from \eqref{n linf ls s<2}, \eqref{lim n 5s 3 5s 3 s<2} and \eqref{lim grad n l2ls}, and the initial estimate  \eqref{reg u0}, we infer 
\begin{equation} \label{tilde n lim}
   {n}_m^k ,\  \tilde{n}_m^k \text{ are bounded in }   L^{\infty}(0, \infty ; L^{s})  \cap L^{5s/3}(0, T ;L^{5s/3}) \cap L^{5s/(s+3)}(0, T ; W^{1,5s/(s+3)}).
\end{equation}
On the other hand, from \eqref{grad z lim} and \eqref{u linfl2 s<2}, and the initial estimate \eqref{reg u0}, one has
\begin{equation} \label{tilde z lim}
    {z}_m^k, \ \tilde{z}_m^k \text{ are bounded in }   L^{\infty}(0, \infty ; L^{\infty}\cap H^1) \cap L^{2}(0, T ; H^2),
\end{equation}
and  
\begin{equation} \label{tilde u lim}
    {u}_m^k,\ \tilde{u}_m^k,\ \hat{u}_m^k \text{ are bounded in }   L^{\infty}(0, \infty ; H) \cap L^{2}(0, T ; V).
\end{equation}

To estimate $\partial_t \tilde{u}_m^k$ from  equation $\eqref{sistema edp}_3$, we observe that $\hat{u}_m^k \otimes u_m^k $ and $ n^k_m $ are bounded in $ L^{5/3}(0,T;L^{5/3}) $,
thus
\begin{equation} \label{ut lim l5/3 w1 3}
    \partial_t \tilde{u}_m^k \text{ is bounded in } L^{5/3}(0,T;(W^{1,5/2}_{0,\sigma})').
\end{equation}

\subsection{Passing the limit as \texorpdfstring{$(m,k) \to (\infty,0)$}{(m,k) -> (inf,0)} for the case  \texorpdfstring{$s \in (1,2)$} { s e (1,2)} }
Having all the necessary estimates, we can start looking for strong convergence to pass to the limit in nonlinear terms. 

With respect to  $\{\tilde{n}^k_m\}$ sequence, from estimates  \eqref{tilde n lim}, \eqref{d-t-n}, and  the compactness Lemma \ref{Lema Aubin Lions}, there exist a subsequence of $\{ \tilde{n}^k_m \}$, relabeled the same, and a limit function $n$ such that, as $(m,k) \to (\infty,0)$,
%
\begin{equation*} 
   \tilde{n}^k_m  \to n \text{ in } C([0,T];(W^{1,10})')\cap L^{5s/(s+3)}(0, T ;L^{5s/(s+3)}).
\end{equation*}
By accounting \eqref{conv n ntil}, we also have the same convergences for $n^k_m$ 
towards the same limit $n$. Moreover, from  \eqref{tilde n lim},
\begin{equation} \label{n-tilde}
   {n}^k_m   
    \to   n \text{ weakly in }  L^{5s/3}(0, T ; L^{5s/3}) \cap L^{5s/(s+3)}(0, T ; W^{1,5s/(s+3)}), 
    \text{ weakly$*$ in } L^\infty(0,\infty;L^s), 
\end{equation}
and
\begin{equation} \label{n conv lplp}
   n^k_m  \to n \text{ in } L^{p}(0,T;L^{p}), \quad \forall\, p \in [1,5s/3),
\end{equation}
and from \eqref{d-t-n}
\begin{equation} 
\partial_t \tilde{n}^k_m \to  \partial_t n \text{ weakly in }  L^{10/9}(0,T;(W^{1,10})').
 \end{equation}

By applying the Dominated Convergence Theorem jointly to \eqref{n conv lplp},
\begin{equation} \label{Tm conv}
     T^m(n_m^k) \to n \text{ in } L^p\left(0, T ; L^p\right), \quad \forall \, p \in[1,5 s / 3).
\end{equation}
Similarly, 
again by  Dominated Convergence Theorem, \eqref{n conv lplp} and Lemma \ref{lema conv raiz},  it follows that
\begin{equation} \label{Gm conv}
    G^m(n_m^k) \to \frac{n^s}{s} \text{ in }  L^p\left(0, T ; L^p\right), \quad \forall\, p \in [1,5 / 3). 
\end{equation}

Arguing now  for $\{\tilde{z}_m^k\}$, using estimates \eqref{tilde z lim} and \eqref{zt lim l2l3/2},  
from the compactness of Lemma \ref{Lema Aubin Lions} 
we get the strong convergence, up to subsequences, and  a limit function $z$ such that 
\begin{equation} \label{z conv l2h1}
    \tilde{z}_m^k \to z \text{ in } C([0,T];L^2)\cap L^2(0,T;H^1)\cap  L^p(0,T;L^p), \ \forall p \in [1,\infty).
\end{equation}
Due to \eqref{conv cont pra por partes}, the same convergences will hold true to $z_m^k$  towards the same limit function $z$. Also, from estimate \eqref{grad z lim} 
we have, up to a subsequence, 
\begin{equation} \label{z fraco estrela}
    z_m^k \overset{\ast} \to 
     z  \text{ weakly$*$ in } L^\infty(0,\infty;H^1 \cap L^\infty)  \text{ and weakly in } L^2(0,T;H^2),
\end{equation}
by \eqref{grad z linf l2 s<2} 
\begin{equation} \label{grad z conv fraca l4l4}
      \nabla z_m^k \to
      \nabla z \text{ weakly in } L^4(0,T;L^4),
\end{equation}
and by \eqref{zt lim l2l3/2} 
\begin{equation} \label{zt conv fraca l2 l3/2}
      \partial_t \tilde{z}_m^k \to
      \partial_t z \text{ weakly in } L^{5/3}(0,T;L^{5/3}).
\end{equation}

With respect to velocity sequence $\{ \tilde{u}^k_m \}$, we use estimates \eqref{tilde u lim} and \eqref{ut lim l5/3 w1 3}, and that $V$ is compactly embedded in $H$ and $H$ is continuously embedded in $(W^{1,5/2}_{0,\sigma})^\prime$. Therefore, by Lemma \ref{Lema Aubin Lions}, 
 there exist a subsequence of $\{ \tilde{u}^k_m \}$, again relabeled the same, and  a limit function $u$ such that 
\begin{equation} \label{u conv l2l2}
    \tilde{u}_m^k \to u \text{ in } C([0,T];(W^{1,5/2}_{0,\sigma})')\cap L^{2}(0,T;H)\cap L^{p}(0,T;L^{p}), \quad \forall p<10/3. 
\end{equation}
%
%
%
%
Again, by \eqref{conv cont pra por partes}, those convergences will also hold for $u^k_m$. 
Moreover, from \eqref{u linfl2 s<2} and \eqref{u l2l6 lim 10/3 10/3}
 \begin{equation} \label{u conv fraca 10/3 10/3}
     u^k_m \to
     u \text{ weakly$*$ in } L^{\infty}(0,\infty;H), 
     \text{ weakly in } L^{2}(0,T;V) \cap L^{10/3}(0,T;L^{10/3}),
 \end{equation}
and from \eqref{ut lim l5/3 w1 3}
\begin{equation}
    \partial_t \tilde{u}_m^k  \to \partial_t u \text{ weakly in } L^{5/3}(0,T;(W^{1,5/2}_{0,\sigma})').
\end{equation}

We now handle the nonlinear terms. Since $g(x) = x^{-1}$ is a Lipschitz function on $[\alpha,\infty)$, from \eqref{z conv l2h1}, one has 
\begin{equation} \label{1/z conv lplp}
    \frac{1}{z^k_m} \to \frac{1}{z} \text{ in } L^p(0,T;L^p), \quad \forall \ p <\infty,
\end{equation}
%
$$
 \nabla (z_m^k)^2 \to \nabla( z^2) \text{ in } L^p(0,T;L^p),\quad \forall\, p<4,
$$
%
and from 
\eqref{tm grad z2 l2 ls}
\begin{equation} \label{Tmn gradz2 conv fraca l20l20}
    T^m(n_m^k) \nabla(z_m^k)^2 \to
    n \nabla (z^2) \ \text{ weakly in } L^{5s/(s+3)}(0,T;L^{5s/(s+3)}).
\end{equation}

Arguing as above, from \eqref{z conv l2h1} and \eqref{grad z conv fraca l4l4} 
\begin{equation} \label{grad z 1/z conv fraca l2 l2}
    \frac{\abs{\nabla z_m^k}^2}{z_m^k} \to
    \frac{\abs{\nabla z}^2}{z}
      \ \text{ weakly in } L^{2}(0,T;L^{2}),
\end{equation}
from \eqref{Gm conv} and \eqref{z conv l2h1} 
\begin{equation} \label{Tms z conv l5/3 l5/3}
    G^m( n_m^k) z_m^k \to 
    \frac{1}{s}n^s z  \ \text{ weakly  in } L^{5/3}(0,T;L^{5/3}),
\end{equation}
and from  \eqref{Gm conv} and  \eqref{1/z conv lplp} 
\begin{equation}\label{Tms 1/z conv l5/3 l5/3}
      \frac{s \, G^m( n_m^k)}{z_m^k} \to
    \frac{n^s}{z}  \ \text{ weakly  in } L^{5/3}(0,T;L^{5/3}).
\end{equation}

Moreover, from  \eqref{grad z conv fraca l4l4}  and \eqref{u conv l2l2},  
\begin{equation} \label{ u grad z conv l2l2 }  
u_m^k \cdot \nabla z_m^k \to u \cdot \nabla z \ \text{ weakly  in } L^{20/11}(0,T;L^{20/11}).
\end{equation}

As for the convective term on $n_m^k$, note that due to \eqref{n conv lplp}, \eqref{u conv fraca 10/3 10/3} 
and \eqref{un 10s10s}, we get that
\begin{equation}\label{un conv 10s 10s}
     u^k_m n^k_m   \to
     u\, n \ \text{ weakly in } L^{10s/(3s+6)}(0,T;L^{10s/(3s+6)}).
\end{equation}

For the fluid convective term, from \eqref{u conv l2l2} and \eqref{u conv fraca 10/3 10/3}, 
\begin{equation} \label{u times u conv l1 l1}
    \hat{u}_m^k \otimes  u_m^k \to u \otimes u \text{ weakly in } L^{5/3}(0,T;L^{5/3}).
\end{equation}

By accounting all previous convergences, it is possible to pass the limit in each term of problem \eqref{sistema edp} as $(m,k) \to (\infty,0)$ and obtain that the limit $(n,z,u)$ is a  weak solution of system 
\begin{equation}   \label{Sistema com z}
 \left\{ \ \begin{array}{l}
 n_t+ u \cdot \nabla n  =\Delta n-\nabla \cdot(n\nabla (z^2)), \\
 \displaystyle z_t  + u \cdot \nabla z= \Delta z- \frac{1}{2}n^s\left(z - \frac{\alpha^2}{z} \right) +\frac{\abs{\nabla z}^2}z, \\
 u_t +u \cdot \nabla u  +\nabla P  =\Delta u + n\nabla \Phi,  \\
 \nabla \cdot u  =0, \\ 
  \left.\partial_\eta n\right|_{\Gamma}=\left.\partial_\eta z \right|_{\Gamma}= \left. u\right|_{\Gamma}= 0 , \quad n(0) = n^0, \quad z(0) = z^0, \quad u(0) = u^0.
\end{array}\right.  
\end{equation}
Note that the initial conditions are also well-defined  in $L^s\times (H^1\cap L^\infty)\times H$ due to the weak continuity in time of the solution $(n,z,u)$  in this space. 

As systems \eqref{Sist cont} and \eqref{Sistema com z} are equivalent in the sense that  
$ (n, z,u) $ is a weak solution of \eqref{Sistema com z} if and only if  $ (n, c,u)$  is a weak solution of \eqref{Sist cont}, with $ c = z^2 -\alpha^2$, we can deduce the existence of a weak solution  $ (n, c,u)$ to system \eqref{Sist cont}.


 \subsection{Passage to the limit as \texorpdfstring{$(m,k) \to (\infty,0)$}{(m,k) -> (inf,0)} for 
 the case \texorpdfstring{$s \geq 2$}{ s >= 2}.}
We first consider the energy estimates, and then proceed to get flux estimates. 
 
 \subsection*{Energy estimates}
   From Lemma \ref{en s>2}, 
  we obtain all  estimates from the previous case $s< 2$, and the additional estimate
    \begin{equation} \label{grad n lim l2 ls2}
   \nabla (n_m^k)^{s/2} \text{ is bounded in } L^2(0,T;L^2),
    \end{equation}
which allows us to simplify the argument to gain the uniform in time energy estimates given in the proof of Lemma \ref{energy-estim s<2}. Indeed, now it is enough to consider 
\[a_i :=  \frac{1}{4(s-1)} \norm{n^i}_{L^s}^s +  \frac{1}{2}\norm{\nabla z^i}_{L^2}^2 + \frac{C_1}{2}\norm{u^i}_{L^2}^2,\]
and
\[
    d_i :=  \dfrac{1}{4s}\norm {\nabla(n^i)^{s/2}}_{L^2}^2
    +  C_2 \norm{\nabla u^i}_{L^2}^2 +C_3 \left( \norm{D^2 z^i}_{L^2}^2 
    + \int \frac{\abs{\nabla z^i}^4}{(z^i)^2} \right)
     + \dfrac{s}{4} \int  G^m(n^i) \abs{\nabla z^i}^2 .
\]
 We just notice that instead of \eqref{est n+1a} and \eqref{est n+1}, 
 now the argument is
$$
      \norm{n^i}_{L^s}^s =      \norm{(n^i)^{s/2}}_{L^2}^2 \leq C\norm{\nabla (n^i)^{s/2}}_{L^2}^2 + C \norm{(n^i)^{s/2}}_{L^1}^2,
$$
and for the second term on the RHS, using interpolation, 
 there exists  $a\in (0,1)$ such that 
   \begin{align*}
    \norm{(n^i)^{s/2}}_{L^1} ^2   & = \norm{n^i}^{s}_{L^{s/2}} \leq \norm{n^i}_{L^s}^{sa} \norm{n^i}_{L^1}^{s(1-a)} 
            \leq \epsilon \norm{n^i}^{s}_{L^s} + C(\norm{n^0}_{L^1},\epsilon),
       \end{align*}
for any $\epsilon>0$ small enough.

\subsection*{Flux estimates} 
Following a similar idea stated in \cite{Correa_Vianna_Filho2023-gz}, an estimate of the chemotactic flux $ T^m(n_m^k) \nabla z_m^k $ can be obtained by splitting the domain in terms of $n^k_m(t,\cdot)$ 
by defining the sets $$
\left\{0 \leq n_m^k(t,\cdot) \leq 1\right\}=\left\{x \in \Omega \mid 0 \leq n_m^k\left(t, x\right) \leq 1
 \right\}
$$
and
$$
\left\{n_m^k (t,\cdot)> 1\right\}=\left\{x \in \Omega \mid n_m^k\left(t, x\right) > 1
\right\} .
$$
It holds that, a.e. in $(0,T)$
\begin{align*}
    \int T^m(n_m^k )^2 \abs{\nabla z_m^k}^2  & = \int_{\left\{0 \leq n_m^k \leq 1\right\}} T^m(n_m^k )^2 \abs{\nabla z_m^k}^2  + \int_{\left\{ n_m^k \geq 1\right\}} T^m(n_m^k) ^2 \abs{\nabla z_m^k}^2  \\ &\leq \int_{\left\{0 \leq n_m^k \leq 1\right\}}  \abs{\nabla z_m^k}^2  + \int_{\left\{ n_m^k \geq 1\right\}} T^m(n_m^k )^s \abs{\nabla z_m^k}^2   \\ & \leq \int  \abs{\nabla z_m^k}^2  + \int T^m(n_m^k )^s \abs{\nabla z_m^k}^2 .
\end{align*}
By integrating in time
\begin{align*}
    \int_0^T \int T^m(n_m^k)^2 \abs{\nabla z_m^k}^2  & \leq \int_0^T \int  \abs{\nabla z_m^k}^2 + \int_0^T\int T^m(n_m^k)^s \abs{\nabla z_m^k}^2,  
\end{align*}
which is bounded by  \eqref{z linf l2 s<2} and  \eqref{lim tm s grad z l2l2}.
Therefore,
\begin{equation} \label{tm n nabla z lim l2 l2}
    T^m(n_m^k) \nabla z_m^k \text{ is bounded in } L^2 (0,T;L^2),
\end{equation}
hence passing to the limit one arrives at $n \nabla z \in L^2 (0,T;L^2)$.

Estimate \eqref{tm n nabla z lim l2 l2} will help us to obtain the $L^2 (0,T;L^2)$ estimate for the diffusion flux $\nabla n_m^k$. Indeed, by testing 
$n^i$ equation  by $n^i$, 
we get
\begin{align*}
   \delta_t\frac{1}{2}\int \abs{n^i}^2 &+  \frac{1}{2k} \int \abs{n^i - n ^{i-1}}^2 + \int \abs{\nabla n^i}^2 
     = 2\int z^i \,T^m(n^{i}) \nabla z^i \cdot \nabla n^i \\ &\leq 2 \norm{z ^i}_{L^\infty} \norm{T^m (n^i) \nabla z^i}_{L^2} \norm{ \nabla n^i}_{L^2} 
     \leq
  C\int T^m (n^i)^2  \abs{\nabla z^i}^2 + \frac{1}{2}  \int \abs{\nabla n^i}^2.
\end{align*}
Then multiplying by $2k$ and summing in $i$, and using \eqref{tm n nabla z lim l2 l2}, we end up with
\begin{equation}\label{grad n lim l2l2}
    \nabla n_m^k \text{ is bounded in } L^2(0,T;L^2),
\end{equation}
hence passing to the limit one arrives at $ \nabla n \in L^2 (0,T;L^2)$. 

As in the case $s<2$, we can conclude that the limit functions $n,z,$ and $u$ are a weak solution to the system \eqref{Sistema com z}, completing the proof of Theorem \ref{resultado principal}.

\section*{Acknowledgments}

D. Barbosa was financed in part by the Coordenação de Aperfeiçoamento de Pessoal de Nível Superior - Brasil (CAPES) - Finance Code 001 and CNPq-Brazil. 
F. Guill\'en-Gonz\'alez was partially supported by Grant I+D+I PID2023-149182NB-I00 funded by 
MICIU/AEI/10.13039/501100011033 and ERDF/EU. 
G. Planas was partially supported by CNPq-Brazil grant 310274/2021-4, and FAPESP-Brazil grant 19/02512-5

 \bibliographystyle{plain}
\bibliography{referencias.bib} 
\end{document}